\PassOptionsToPackage{dvipsnames}{xcolor}
\documentclass[onefignum,onetabnum,twoside]{siamart220329}

\usepackage[T1]{fontenc}
\usepackage[utf8]{inputenc}
\usepackage[american]{babel}

\usepackage[algo2e,vlined]{algorithm2e}
\usepackage{algpseudocode}
\usepackage{amsmath}
\usepackage{amssymb}
\usepackage{booktabs}
\usepackage{dsfont}
\usepackage{float}
\usepackage{hyperref}
\usepackage{mathtools}
\usepackage{mdframed}
\usepackage[square,numbers,sort&compress]{natbib}
\usepackage{subfig}
\usepackage{microtype}

\usepackage{tikz}
\usetikzlibrary
  {arrows,calc,through,backgrounds,matrix,positioning,decorations.pathmorphing,
shapes.misc}

\usepackage{gnuplot-lua-tikz}

\newsiamremark{remark}{Remark}
\newsiamremark{example}{Example}
\newsiamremark{assumption}{Assumption}

\newcommand{\eps}{\varepsilon}
\newcommand{\eff}{\varepsilon^{\text{eff}}}

\renewcommand{\vec}[1]{\boldsymbol{#1}}
\newcommand{\vH}{\vec H}
\newcommand{\vE}{\vec E}
\newcommand{\vJ}{\vec J}
\newcommand{\vn}{\vec \nu}
\newcommand{\vt}{\vec \tau}
\newcommand{\hvn}{\vec{\hat\nu}}
\newcommand{\hvt}{\vec{\hat\tau}}
\newcommand{\vx}{\vec x}
\newcommand{\vy}{\vec y}
\newcommand{\hvy}{\vec{\hat y}}
\newcommand{\hz}{\vec{\hat z}}
\newcommand{\vq}{\vec q}
\newcommand{\hvq}{\vec {\hat q}}

\newcommand{\hvchi}{\vec{\hat \chi}}
\newcommand{\hvvarphi}{\vec{\hat \varphi}}

\newcommand{\calH}{\vec{\mathcal{H}}}
\newcommand{\calE}{\vec{\mathcal{E}}}

\newcommand{\im}{\textrm i}
\newcommand{\unit}{\vec I}

\newcommand{\dy}{\,\mathrm{d}y}
\newcommand{\doy}{\,\mathrm{d}o_y}
\newcommand{\dhy}{\,\mathrm{d}\hat y}
\newcommand{\dhohy}{\,\mathrm{d}{\hat o}_{\hat y}}

\newcommand{\tdiffop}{\hat\nabla^h_{\eta}}

\newcommand{\tnegdiffop}{\hat\nabla^{-h}_{\eta}}

\newcommand{\veta}{{\vec\eta}}

\newcommand{\tangent}[1]{#1_{\mathbf{t}}}
\newcommand{\htangent}[1]{#1_{\hat{\mathbf{t}}}}

\newcommand{\tnabla}{\tangent{\nabla}}
\newcommand{\htnabla}{\htangent{\hat\nabla}}
\newcommand{\tve}[1]{\tangent{{\vec e_{#1}}}}

\renewcommand{\Re}{\mathrm {Re}\,}
\renewcommand{\Im}{\mathrm {Im}\,}

\newcommand{\C}{\mathbb C}
\newcommand{\R}{\mathbb R}

\newcommand{\Z}{\mathbb Z}

\newcommand{\Si}{{\Sigma}}

\newcommand{\js}[1]{\left[#1\right]_{\Si}}

\newcommand{\ie}{i.\,e.}

\headers%
  {Shape Optimal Optical Microstructures}%
  {Bezbaruah, Maier, Wollner}

\title{Shape optimization of optical microscale inclusions}

\author{%
  Manaswinee Bezbaruah%
  \thanks{Department of Mathematics, Texas A\&M University, 3368 TAMU,
    College Station, TX 77843, USA (\email{\{bezba004,maier\}@tamu.edu})}
  \and%
  Matthias Maier\footnotemark[1]
  \and%
  Winnifried Wollner%
  \thanks{
    Fachbereich Mathematik, MIN Fakultät, Universität Hamburg,
    Bundesstr. 55, 20146 Hamburg, Germany
    (\email{winnifried.wollner@uni-hamburg.de})}
}

\begin{document}

\maketitle
\begin{abstract}
  This paper describes a class of shape optimization problems for optical
  metamaterials comprised of periodic microscale inclusions composed of a
  dielectric, low-dimensional material suspended in a non-magnetic bulk
  dielectric. The shape optimization approach is based on a homogenization
  theory for time-harmonic Maxwell's equations that describes effective
  material parameters for the propagation of electromagnetic waves through
  the metamaterial. The control parameter of the optimization is a
  deformation field representing the deviation of the microscale geometry from
  a reference configuration of the cell problem. This allows for describing the 
  homogenized effective permittivity tensor as a function of the
  deformation field.
  We show that the underlying deformed cell problem is well-posed and
  regular. This, in turn, proves that the shape optimization problem is
  well-posed. In addition, a numerical scheme is formulated that utilizes an
  adjoint formulation with either gradient descent or BFGS as optimization
  algorithms. The developed algorithm is tested numerically on a number of
  prototypical shape optimization problems with a prescribed effective
  permittivity tensor as the target.
\end{abstract}




\section{Introduction}
\label{sec:introduction}
Plasmonic crystals consisting of 2D dielectric inclusion have given way to
a wide range of striking optical phenomena that are unlike the classical
behavior of electromagnetic waves. Some examples of such optical phenomenon
are optical cloaking, negative-refraction, subwavelength focusing, and the
\emph{epsilon-near-zero} effect~\cite{alu2005, nemilentsau2016, cheng2014,
mattheakis2016, huang2011, moitra2013, silveirinha2007}. A well-studied
example of a plasmonic crystal is graphene nanosheets that are arranged
periodically with subwavelength spacing and suspended in a bulk
non-magnetic dielectric host~\cite{alu2005,cheng2014}. The wide wealth of
applications motivates the need to describe, design, and tune optical
properties of plasmonic crystals~\cite{molesky2018,liberal2017}. This, in
turn, motivates the current work presented here.
Analytic homogenization approaches for time-harmonic Maxwell
equations~\cite{wellander2001, wellander2002, wellander2003} have been
extended successfully to the setting of plasmonic crystals with
lower-dimensional interfaces~\cite{amirat2017, maier19c, maier20c}. Here,
the periodic microstructure is replaced by an effective permittivity tensor
that is determined by a weighted average over a so-called cell problem.
This setup can serve as an efficient computational tool for computing the
optical macroscale response~\cite{maier19c,maier21c}.
In this manuscript, we introduce a shape-optimization approach that is based
on this notion of computing effective material parameters with the help of
microscopic cell problems. Specifically, we introduce a control variable
describing the deformation of a given reference geometry of the cell
problem with the target of achieving a preset macroscopic permittivity
tensor. Our key objectives with the present paper are to
\begin{itemize}
  \item
    Extend the two-scale homogenization result by introducing a deformation
    field. We show, in particular, that the resulting system is well-posed
    and show regularity of the corrector solution of the modified cell problem,
    provided that the deformation field is of suitable regularity.
  \item
    Formulate a shape-optimization problem based on the homogenization
    procedure that uses the deformation field as the control variable. We also
    discuss penalization strategies for ensuring mesh regularity of the
    cell problem.
  \item
    Discuss a number of prototypical numerical results highlighting the
    algorithmic approach. We in particular demonstrate that our shape
    optimization formulation works for manipulating optical microstructures
    to achieve an epsilon-near-zero plasmonic crystal.
\end{itemize}

\subsection{Microscopic model and homogenization theory}
\label{subsec:intro_homogenization}
Figure~\ref{fig:layered} shows the geometry of a three-dimensional
plasmonic crystal consisting of periodic copies of a scaled representative
volume element $Y$ containing a curved surface $\Sigma$ representing a 2D
material with surface conductivity $\sigma$ and otherwise filled with a
dielectric host material. The periodicity is denoted by $d$ and $\Sigma^d$
is the union of all scaled copies of $\Sigma$; see
Figure~\ref{fig:layered}.
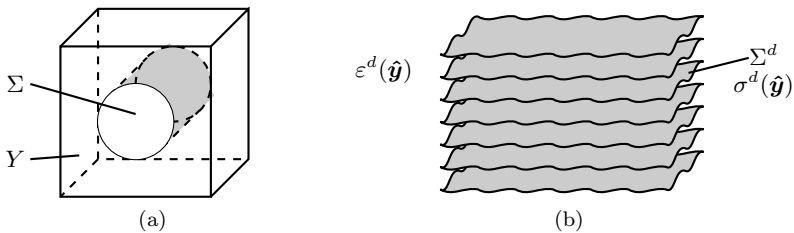
\begin{figure}[h]
  \centering
    \subfloat[]{
      \hspace{-3em}
      \begin{tikzpicture}[scale=1.00]
        \path [thick, draw]
          (-1, -1) -- (1, -1) -- (1, 1) -- (-1, 1) -- cycle;
        \path [thick, draw] (1,-1) -- (1.5, -0.5) -- (1.5, 1.5) -- (1, 1);
        \path [thick, draw] (-1,1) -- (-0.5, 1.5) -- (1.5, 1.5);
        \path [thick, draw, dashed]
          (-0.5, 1.5) -- (-0.5, -0.5) -- (1.5, -0.5);
        \path [thick, draw, dashed] (-0.5, -0.5) -- (-1.0, -1.0);
        \path [thick, draw, dashed] (0.5, 0.5) circle (0.5);
        \path [thick, draw] (0.0, 0.0) circle (0.5);
        \path [thick, draw, dashed] (0.35, -0.35) -- (0.85, 0.15);
        \path [thick, draw, dashed] (-0.35, 0.35) -- (0.15, 0.85);

        \path [fill= black, fill opacity=0.2]
        (1.05, 0.35) -- (0.45, 1.05) -- (-0.25, 0.45) -- (0.25, -0.45)
        (0.5, 0.5) circle (0.5);
        \path [fill= white, fill opacity=1.0](0.0, 0.0) circle (0.5);
        \path [thick, draw] (-1.4, 0.5) -- (0.0, 0.1);
        \node at (-1.6, 0.5) {\small $\Sigma$};
        \path [thick, draw] (-1.4, -0.5) -- (-0.75, -0.45);
        \node at (-1.6, -0.5) {\small $Y$};
    \end{tikzpicture}}
    \hspace{3em}
    \subfloat[]{
      \begin{tikzpicture}[scale=1.00]
      \foreach \x in {0,...,6} {
        \path [fill=white]
          (-1.5, -1+0.3*\x) to[out=-30,in=170] (-1.0, -1+0.3*\x)
                             to[out=-30,in=170] (-0.5, -1+0.3*\x)
                             to[out=-30,in=170]  (0.0, -1+0.3*\x)
                             to[out=-30,in=170]  (0.5, -1+0.3*\x)
                             to[out=-30,in=170]  (1.0, -1+0.3*\x)
                             to[out=-30,in=170]  (1.5, -1+0.3*\x)
                             to[out=-30,in=170] (1.75, -0.75+0.3*\x)
                             to[out=-30,in=170] (2.00, -0.5+0.3*\x)
                             to[out=170,in=-30] (1.50, -0.5+0.3*\x)
                             to[out=170,in=-30] (1.00, -0.5+0.3*\x)
                             to[out=170,in=-30] (0.50, -0.5+0.3*\x)
                             to[out=170,in=-30] (0.00, -0.5+0.3*\x)
                             to[out=170,in=-30] (-0.5, -0.5+0.3*\x)
                             to[out=170,in=-30] (-1.0, -0.5+0.3*\x)
                             to[out=170,in=-30] (-1.25, -0.75+0.3*\x)
                             to[out=170,in=-30] cycle;
        \path [thick, draw, fill=black, fill opacity=0.2]
          (-1.5, -1+0.3*\x) to[out=-30,in=170] (-1.0, -1+0.3*\x)
                             to[out=-30,in=170] (-0.5, -1+0.3*\x)
                             to[out=-30,in=170]  (0.0, -1+0.3*\x)
                             to[out=-30,in=170]  (0.5, -1+0.3*\x)
                             to[out=-30,in=170]  (1.0, -1+0.3*\x)
                             to[out=-30,in=170]  (1.5, -1+0.3*\x)
                             to[out=-30,in=170] (1.75, -0.75+0.3*\x)
                             to[out=-30,in=170] (2.00, -0.5+0.3*\x)
                             to[out=170,in=-30] (1.50, -0.5+0.3*\x)
                             to[out=170,in=-30] (1.00, -0.5+0.3*\x)
                             to[out=170,in=-30] (0.50, -0.5+0.3*\x)
                             to[out=170,in=-30] (0.00, -0.5+0.3*\x)
                             to[out=170,in=-30] (-0.5, -0.5+0.3*\x)
                             to[out=170,in=-30] (-1.0, -0.5+0.3*\x)
                             to[out=170,in=-30] (-1.25, -0.75+0.3*\x)
                             to[out=170,in=-30] cycle;
      }
      \node at (-2.25, +0.6) {\small $\varepsilon^d(\hvy)$};
      \node at (2.8, +0.8) {\small $\Sigma^d$};
      \path [thick, draw] (2.6, 0.75) -- (1.8, 0.55);
      \node at (2.8,  0.4) {\small $\sigma^d(\hvy)$};
    \end{tikzpicture}}

  \caption{(a) The unit cell $Y = [0,1]^3$ consisting of 2D
    graphene inclusions $\Sigma$ with surface conductivity $\sigma$
    in an ambient host material with permittivity $\varepsilon$; (b)
    the plasmonic crystal formed by many scaled and repeated
    copies of $Y$ in all space directions.
    \label{fig:layered}}

\end{figure}
We assume that the dielectric host has a uniform
and isotropic relative permittivity $\eps$ and permeability $\mu$. The
(scaled) surface conductivity $\sigma^d$ in the numerical tests will be
given by a simple Drude model; see Section~\ref{subsec:maxwell}. The
time-harmonic response of the plasmonic crystal is described by the
time-harmonic Maxwell equations governing the scattering of an
electromagnetic wave $(\vE^d,\vH^d)$, viz.
\begin{align*}
    \nabla \times \vE^d = i\omega \mu\vH^d,
    \qquad
    \nabla \times \vH^d = -i\omega\eps \vE^d + \vJ^d_a.
\end{align*}
Here, $\omega$ is the angular frequency, $\mu$ denotes the relative
permeability and $\vJ^d_a$ is a prescribed current density. This system of
equations is now closed with jump conditions over the interface $\Sigma^d$
that arise from the surface conductivity $\sigma^d(\vx)$; see
Section~\ref{subsec:maxwell}. In the limit $d\to0$ for the periodicity, the
time-harmonic system reduces to a \emph{homogenized system},
\begin{align*}
    \nabla \times \calE = i\omega \mu\calH,
    \qquad
    \nabla \times \calH = -i\omega\eff \calE + \vJ^d_a,
\end{align*}
where the heterogeneous permittivity and microscale inclusions have been
replaced by an effective permittivity tensor $\eff(\vx)$ that is given by a
\emph{weighted} average~\cite{amirat2017,maier20c} over the unit cell $Y$.
The weight is due to a \emph{corrector} $\vec\chi$ that is described by an
auxiliary cell problem incorporating the effects of the microscale
inclusions $\Sigma$; see Section~\ref{subsec:homogenization}.

\subsection{Shape optimization approach}
\label{subsec:intro_deformation}
We show in Section~\ref{subsec:well_pos} that the above homogenization
result establishes a smooth functional relationship between the effective
permittivity tensor $\eff$ and the material parameters $\eps$, $\mu$, $\sigma^d$,
as well as the hypersurface $\Sigma$. We exploit this functional
relationship to formulate a shape optimization problem in which we
deform $\Sigma$ for minimizing the misfit of $\eff$ and a preset target
permittivity tensor $\eps^{\text{target}}$. To this end, we introduce a
deformation field $\hvq:\hat Y\to \R^3$ that acts on a reference
configuration $\hat Y$ with inclusion $\hat \Sigma$, see Figure~\ref{fig:deformation}:
\begin{align*}
  \hat Y \ni \hvy
  \quad\leftrightarrow\quad
  \vy(\hvy) := \hvy + \hvq(\hvy) \in Y.
\end{align*}
\begin{figure}[tb]
  \centering

  \begin{tikzpicture}
    \path [thick, draw]
      (-1, -1) -- (1, -1) -- (1, 1) -- (-1, 1) -- cycle;
    \path [thick, draw] (1,-1) -- (1.5, -0.5) -- (1.5, 1.5) -- (1, 1);
    \path [thick, draw] (-1,1) -- (-0.5, 1.5) -- (1.5, 1.5);
    \path [thick, draw, dashed]
      (-0.5, 1.5) -- (-0.5, -0.5) -- (1.5, -0.5);
    \path [thick, draw, dashed] (-0.5, -0.5) -- (-1.0, -1.0);
    \path [thick, draw, dashed] (0.5, 0.5) circle (0.5);
    \path [thick, draw] (0.0, 0.0) circle (0.5);
    \path [thick, draw, dashed] (0.35, -0.35) -- (0.85, 0.15);
    \path [thick, draw, dashed] (-0.35, 0.35) -- (0.15, 0.85);
    \path [very thick, draw, ->] (2, 0.0) -- (3.5, 0.0);
    \path [fill= black, fill opacity=0.2]
    (1.05, 0.35) -- (0.45, 1.05) -- (-0.25, 0.45) -- (0.25, -0.45)
    (0.5, 0.5) circle (0.5);
    \path [fill= white, fill opacity=1.0](0.0, 0.0) circle (0.5);
    \node at (2.7, -0.28){\footnotesize $\hvq$};
    \path [thick, draw]
      (4, -1) -- (6, -1) -- (6, 1) -- (4, 1) -- cycle;
    \path [thick, draw] (6,-1) -- (6.5, -0.5) -- (6.5, 1.5) -- (6, 1);
    \path [thick, draw] (4,1) -- (4.5, 1.5) -- (6.5, 1.5);
    \path [thick, draw, dashed]
    (4.5, 1.5) -- (4.5, -0.5) -- (6.5, -0.5);
    \path [thick, draw, dashed] (4.5, -0.5) -- (4.0, -1.0);
    \draw[rounded corners=5mm,thick]
    plot coordinates{(4.65,1) (5.5,0.5) (6,-0.25) (4.5,-0.5)}--cycle;
    \draw[rounded corners=5mm,thick,dashed]
    plot coordinates{(5.15,1.5) (6,1) (6.5,0.25) (5,0)}--cycle;
    \path [thick, draw, dashed] (5.66, -0.32) -- (6.15, 0.25);
    \path [thick, draw, dashed] (4.65, 0.65) -- (5.15, 1.15);
    \path[rounded corners=5mm,fill = black, opacity =0.2]
    (5.15,1.15) -- (6.35,0.25) -- (5.65,-0.35) -- (4.65, 0.65)
    plot coordinates{(5.15,1.5) (6,1) (6.5,0.25) (5,0)}--cycle;
    \path[rounded corners=5mm,fill = white, opacity =1.0]
    plot coordinates{(4.65,1) (5.5,0.5) (6,-0.25) (4.5,-0.5)}--cycle;
  \end{tikzpicture}
  \caption{A reference unit cell $\hat Y$ with a 2D inclusion
    $\hat\Sigma$ that is deformed by a deformation vector field
    $\hvq(\vx)$.}
  \label{fig:deformation}
\end{figure}
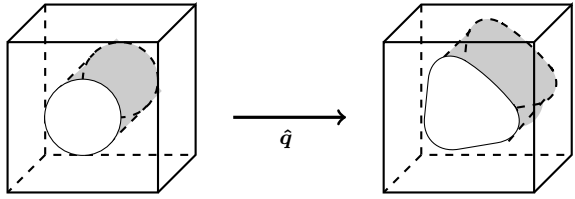
The numerical value of the effective permittivity tensor,
$\eff(\hvchi;\hvq)$, now depends on the \emph{control} $\hvq(\hvy)$ as well
as the \emph{state} $\hvchi(\hvy)$; see~\ref{subsec:ALE_transform}. We are
now in a position to introduce a cost function
\begin{align*}
  C(\hvchi;\hvq) \;:=\; \frac12\,\big\|\eff[\hvchi,\hvq]  -
  \eps^{\text{target}}\big\|_{\text{Fr.}}^2 +
  \frac\alpha2\big\|\nabla\hvq\big\|^2
  + \text{penalization},
\end{align*}
where $\alpha$ is a scaling factor for a Tikhonov regularization term and
the remaining penalization terms for ensuring mesh regularity are discussed
in detail in Section~\ref{sec:optimization}. Here,
$\|\cdot\|$ denotes the $L^2$-norm over $\hat{Y}$,
while $\|\cdot\|_{\text{Fr.}}$ denotes the Frobenius-norm. 

\subsection{Past works}
\label{subsec:past_works}
The use of nanoscale 2D dielectric inclusions for engineering and
manipulating an effective optical response has been discussed extensively
in the physics literature~\cite{alu2005, nemilentsau2016, cheng2014,
mattheakis2016, huang2011, moitra2013, silveirinha2007}. For some recent
advances in optical materials and photonics, we refer the reader to
\cite{tan2020, Hassan2022, yang2022}.

The following shape optimization approach is based on a well-established
periodic homogenization theory for time-harmonic Maxwell
equations~\cite{wellander2001, wellander2002, wellander2003} extended to
inclusions composed of lower-dimensional interfaces~\cite{amirat2017,
maier19c, maier20c}. A number of applications of the homogenization
approach have been discussed for simulating finite-sized
geometries~\cite{maier20e} and describing Lorentz
resonances~\cite{maier21c} of plasmonic crystals.

The minimization of the cost-functional is a particular case of
PDE-constrained optimization, where we utilize adjoint calculus to obtain
efficient formulas for the derivatives of the cost-functional with respect
to the deformations~$\hvq$, see, e.g.,~\cite{Troeltzsch:2010,
HinzePinnauUlbrichUlbrich:2009, BeckerMeidnerVexler:2007, Baubner:2020,
Brandenburg:2009}.

Using homogenization principles for optimal material configurations has
been explored before. For example, in \cite{Frei:2015} the authors use
upscaling techniques to optimize microscale fiber orientation. More
generally, the optimization of optical nanostructures
~\cite{estakhri2015,papadimopoulos2021} and related inverse design
problems~\cite{nikkhah2022, silva2014} have been explored extensively.

Our computational approach uses a
\emph{Broyden-Fletcher-Goldfarb-Shanno} (BFGS) \emph{quasi-Newton
scheme}~\cite{Fletcher:1988}. Such quasi-Newton methods have been
generalized for optimization problems over Riemannian manifolds,
particularly with emphasis on shape optimization, see,
e.g.,~\cite{RingWirth:2012,SchulzSiebenbornWelker:2016,Welker:2017}. To
assert positive definiteness of the iterates these approaches typically
rely on a Powell-Wolfe type step-length selection. As such the step-length
selection relies on a potential increase of the step length which can lead
to conflicts with barrier method we use for the penalization. We thus
utilize a damped BFGS update suggested by~\cite{Powell:1978} for the
inverse Hessian as proposed by~\cite{Wollner:2022}.

\subsection{Paper organization}
\label{subsec:paper_org}
The remainder of the paper is organized as follows.
In Section~\ref{sec:problemdescription}, we summarize the underlying
homogenization result for plasmonic crystals, reformulate the cell problem
and averaging for the case of a deformed geometry, and establish
well-posedness and regularity of the cell problem for the deformed
configuration.
The shape optimization problem is described in Section~\ref{sec:optimization}.
A computational approach and numerical illustration are discussed in
Section~\ref{sec:numerical_examples}, specifically an optimization
algorithm based on gradient descent and an improved version using an
inverse Broyden-Fletcher-Goldfarb-Shanno (BFGS) quasi-Newton method is discussed. We
summarize our work and conclude in Section~\ref{sec:conclusion}. Detailed
proofs of the lemmas and theorems in Section~\ref{sec:problemdescription}
are given in Appendices~\ref{app:estimates} and \ref{app:regularity}.

\section{Problem Description}
\label{sec:problemdescription}

The manuscript is concerned with a class of shape optimization problems
involving optical metamaterials comprised of periodic microscale inclusions.
In this section, we summarize the underlying microscale model and a
corresponding homogenization theory that will serve as a basis for the
shape optimization problem. We then introduce a deformed formulation of the
cell problems and formulate an optimization problem.


\subsection{Background: Heterogeneous Maxwell equations}
\label{subsec:maxwell}
Consider a three-dimensional geometry $\Omega$ consisting of
scaled and periodic copies of a \emph{unit cell} $Y$, which incorporates
microscale inclusions given by 2D material surfaces $\Sigma$; see
Figure~\ref{fig:layered} \cite{maier19c, maier20c}. The time-harmonic
response of an electromagnetic wave $(\vE^d,\vH^d)$ to the scattering
configuration is described by the time-harmonic Maxwell equations valid on
the domain $\Omega\backslash\Sigma^d$, viz.,~\cite{monk03}
\begin{equation}
    \nabla \times \vE^d = i\omega \mu\vH^d,
    \qquad
    \nabla \times \vH^d = -i\omega\eps^d \vE^d + \vJ^d_a.
  \label{eq:maxwell_equations}
\end{equation}
Here, $\vE^d(\vx)$ (and $\vH^d(\vx)$) denote the time-harmonic,
complex-valued electric (and magnetic) field component, $\eps^d$ is the
electric permittivity of the ambient host, $\vJ^d_a$ is a given electric
current density, and $\omega$ denotes the angular frequency.
In \eqref{eq:maxwell_equations} we used the convention that fields
exhibit a $e^{-\im\omega t}$ time dependence, meaning, the time-dependent
real-valued physical field $\vec F(\vx,t)$ is reconstructed from its
time-harmonic counterpart $\vec F(\vx)$ by setting $\vec F(\vx,t)=\Re(\vec
F(\vx)e^{-\im\omega t})$.
The parameter $d$ denotes a \emph{scaling parameter} for the microscale
inclusions:
\begin{align*}
  \Sigma^d = \big\{d\,\vec z + d\,\vec \varsigma\;\;:\;\;\vec z \in \Z^3,\;
  \vec \varsigma\in\Sigma\big\}.
\end{align*}
Equation \eqref{eq:maxwell_equations} is now furnished with jump conditions
over the interface $\Sigma^d$ that arise due to the presence of a current
density $\vJ_{\Sigma^d} = \delta_{\Sigma^d} \sigma^d\vE^d_\mathbf{t}$
caused by \emph{surface conductivity} $\sigma^d$ of the microscale
inclusions \cite{maier19c}. Here, $\delta_{\Sigma^d}$ is the surface
measure of $\Sigma^d$. The current density leads to a jump in the
tangential component of the magnetic field. In summary,
\begin{equation}
  \left[\vec \nu \times \vE^d\right]_{\Sigma^d} = 0,
  \qquad
  \left[\vec \nu\times \vH^d\right]_{\Sigma^d} =
  \tangent{(\sigma^d\vE^d)},
  \label{eq:maxwell_jump_conditions}
\end{equation}
where the subscript $\mathbf{t}$ denotes projection onto the tangential
plane of $\Sigma$ and $\left[\,.\,\right]_{\Sigma^d}$ denotes the jump over
$\Sigma^d$ with respect to a chosen normal field $\vn$,%
\begin{align*}
  \left[\vec F\right]_{\Sigma^d}
  := \lim_{\alpha\searrow0}
  \Big(\vec F(\vx+\alpha\vn) - \vec F(\vx-\alpha\vn)\Big)\qquad
  \vx\in\Sigma^d.
\end{align*}
Note that numerical value of the jumps in
\eqref{eq:maxwell_jump_conditions} is independent of the orientation of the
chosen normal field $\vn$ because of the cross product of the fields with
$\vn$.
We will allow the permittivity $\eps^d$ and surface conductivity $\sigma^d$
to be tensor-valued. We make the additional scaling assumption on (bulk)
permittivity and surface conductivity that they depend on a slow scale
($\vx\in\Omega$) and are periodic and rapidly oscillating on a
fast scale ($\vy\in Y$) proportional to $d$, namely
\begin{align*}
  \eps^d(\vx) = \eps(\vx,\vx/d),
  \qquad
  \sigma^d(\vx)\,=\,d\,\sigma(\vx,\vx/d).
\end{align*}
Concretely, in the following we will use a simple Drude
model~\cite{grigorenko2012} for all numerical tests,
\begin{align}
  \sigma(\vx,\vy)
  \,=\, \tilde\sigma(\omega)\,\unit,
  \qquad
  \tilde\sigma(\omega) = \frac{\im\omega_p}{\omega+\im/\tau},
  \qquad
  \omega_p,\,\tau \in \R_{\ge 0}.
  \label{eq:drude}
\end{align}
Equation \eqref{eq:drude} is a common model for the surface conductivity of
graphene~\cite{grigorenko2012}. After a suitable nondimensionalization and rescaling typical
values for the constants are $\omega_p\approx4/137$ and $\tau\approx100$
for frequencies $\omega\approx1$ \cite{maier21c}.

\subsection{Background: Homogenization results}
\label{subsec:homogenization}
System \eqref{eq:maxwell_equations} \& \eqref{eq:maxwell_jump_conditions}
furnished with the Drude model \eqref{eq:drude} exhibits a pronounced
two-scale character with an ambient vacuum wavelength $2\pi/k_0\sim1$ and a
small scale associated with the microscale inclusions and nanoscale
resonances scaling with $d$; typically, $d\ll 2\pi/k_0$
\cite{maier20c,maier21c}. In the limit $d\to0$, system
\eqref{eq:maxwell_equations} \& \eqref{eq:maxwell_jump_conditions} reduces
to a \emph{homogenized system}~\cite{amirat2017, maier19c, maier20c, maier21c,
wellander2001, wellander2002, wellander2003}:
\begin{equation}
    \nabla \times \calE = i\omega \mu\calH,
    \qquad
    \nabla \times \calH = -i\omega\eff \calE + \vJ^d_a,
  \label{eq:homogenized}
\end{equation}
where the heterogeneous permittivity and microscale inclusions have been
replaced by an effective permittivity tensor
\begin{multline}
  \label{eq:eff}
  \eff_{ij}(\vx) = \int_Y \eps(\vx,\vy)
  (\vec e_j+\nabla\chi_j^{\:T})\cdot(\vec e_i+\overline{\nabla\chi_i^{\:T}})\dy
  \\
  - \frac{1}{i\omega} \int_\Sigma \sigma(\vx,\vy)(\tve{j} + \tnabla \chi_j^{\:T})
  \cdot(\tve{i}+\overline{\tnabla\chi_i^{\:T}}) \doy.
\end{multline}
Here, we have adopted the convention that the gradient is a row vector and
$\vec e_i$ denotes the i-th (column) unit vector.
The subscript
$\mathbf{t}$ denotes projection onto the tangential plane of $\Sigma$,
and $\bar{.}$ denotes the complex conjugate of a complex-valued
quantity.
The corrector $\chi_i \in H:=\big\{\varphi\in
H^1_{\text{per}}(Y,\C),\;\tangent{\nabla}\varphi\in \vec
L^2(\Sigma,\C)\big\}$, $i=1,2,3$ is given by an associated cell problem
\begin{multline}
  \label{eq:cell_problem}
  \int_Y \eps(\vx,\vy)
  (\vec e_i+\nabla\chi_i^{\:T})\cdot\overline{\nabla{\varphi}^{\:T}}\dy
  \\
  - \frac{1}{i\omega} \int_\Sigma \sigma(\vx,\vy)
  (\tve{i} + \tnabla \chi_i^{\:T})
  \cdot\overline{\tnabla\varphi^{\:T}}\doy
  \;=\;0\quad\text{for all }\varphi\in H.
\end{multline}
$H^1_{\text{per}}(Y,\C)$ denotes the Sobolev-space of all
complex-valued square-integrable function with square-integrable partial
derivatives defined on the unit cell $Y$ that are periodic in all three
space directions, $\varphi(\vy) = \varphi(\vy +\vec e_i)$ for $i=1,2,3$ and
for all $\vy\in Y$. It has been established in
\cite{amirat2017,maier20c} that \eqref{eq:maxwell_equations} \&
\eqref{eq:maxwell_jump_conditions}, \eqref{eq:homogenized}, and
\eqref{eq:cell_problem} are well-posed provided some assumptions on the
material constants are satisfied. We summarize:
\begin{theorem}[Well-posedness and two-scale convergence
  \cite{amirat2017,maier20c}]
  \label{thm:well_posedness}
  Let $\mu \in \mathbb{R}_{>0}$ and let $\eps^d, \sigma^d \in
  L^\infty(\Omega,\C^{d\times d})$ be bounded, complex and
  tensor-valued functions such that $\Re(\eps^d(\vx))$ and
  $\Im{(\sigma^d(\vx))}$ are symmetric, and $\text{Im}(\eps^d(\vx))$ and
  $\text{Re}{(\sigma^d(\vx))}$ are symmetric and uniformly elliptic. Then,
  provided $\Sigma$ is sufficiently smooth \cite{amirat2017,maier20c},
  \eqref{eq:maxwell_equations} \& \eqref{eq:maxwell_jump_conditions},
  \eqref{eq:homogenized}, and \eqref{eq:cell_problem} are well-posed, the
  solution $(\vE^d,\vH^d)$ \emph{two-scale converges} to the solution
  $(\calE,\calH)$ of the homogenized problem \eqref{eq:homogenized}.
\end{theorem}
\begin{remark}
  The above theorem is, strictly speaking, a generalized version of what has
  been shown in \cite{amirat2017,maier20c} where the material parameters
  $\eps^d$ and $\sigma^d$ had been assumed to be scalar (and not
  tensor-valued coefficients). Proving the generalized statement, however,
  only requires minor adjustments to the proofs.
\end{remark}

\subsection{Deformed cell problem}
\label{subsec:ALE_transform}

We now introduce a deformation of the cell problem. For this we let $\hat
Y$ and $\hat \Sigma$ denote a given reference configuration consisting of
the unit cell $\hat Y = [0,1]^3$ with given (two dimensional) inclusions
$\hat \Sigma$. Let $Y$ and $\Sigma$ denote the deformed volume and
inclusions. We assume that the deformed geometry is given by a deformation
vector field $\hvq(\hvy)\in \vec D(\hat Y, \hat\Sigma)$
of the reference configuration \cite[pp. 70-79]{richter2017},
such that
\begin{align}
  \hat Y \ni \hvy
  \;\leftrightarrow\;
  \vy(\hvy) := \hvy + \hvq(\hvy) \in Y
  \label{eq:deformation}
\end{align}
is a bijection that also maps $\hat\Sigma$ onto $\Sigma$
bijectively. We adopt the notation that $\hat .$ indicates that a
function, coordinate or differential operator is on the reference
configuration \cite{richter2017}. Here, in analogy to the function
space $H$ we have introduced
\begin{align}
  \vec D(\hat Y, \hat\Sigma) := \big\{\hvvarphi\in \vec W^{1,\infty}_0(\hat
  Y): \;\hvt_k\cdot \hat\nabla\hvvarphi\in L^\infty(\hat
  \Sigma),\,k=1,2\big\},
  \label{eq:control_space}
\end{align}
where $W^{1,\infty}_0(\hat Y)$ shall denotes the Sobolev space of all functions
$\varphi\in W^{1,\infty}_0(\hat Y)$ with vanishing trace on $\partial\hat
Y$.
\begin{remark}
  We note that $\hvq\in\vec D(\hat Y, \hat\Sigma)$ implies that $\hvq$ is
  already Lipschitz continuous; see, for example,
  \cite[Ex.\,8,\,p.\,156]{LiebLoss:2001}. This implies that the trace of
  $\hvq$ onto $\Sigma$ and $\partial\Omega$ is well defined and also
  Lipschitz continuous. Strictly speaking, this renders the condition
  $\hvt_k\cdot \hat\nabla\hvvarphi\in L^\infty(\hat \Sigma)$ superfluous.
  \label{rem:control_space}
\end{remark}
\begin{definition}
  \label{defi:transformation_gradient}
  The \emph{transformation gradient} $\hat F(\hvy)$ and determinant $\hat
  J$ are given by
  \begin{align}
    \label{eq:transformation_gradient}
    \hat F(\hvy)\;:=\;\unit+\hat\nabla\hvq(\hvy),
    \qquad
    \hat J(\hvy)\;:=\;\det\big(\hat F(\hvy)\big).
  \end{align}
\end{definition}
Here, we have adopted the convention that
$\big(\hat\nabla\hvq\big)_{ij}=\tfrac{\partial}{\partial{\hat y_j}}\hat
q_i$ and $\unit$ denotes the unit matrix. We then have,
\begin{lemma}[Transformation]
  \label{lem:transformation}
  Let $\hat\varphi(\hvy)$ be a differentiable function defined on $\hat Y$
  and let $\varphi(\vy)\;:\;Y\to\R$ be defined by setting
  $\varphi(\vy):=\hat\varphi(\hvy(\vy))$. Then,
  \begin{align}
    \label{eq:transformed_gradient}
    \nabla\varphi(\vy)^{\,T}
    \;=\;\hat F(\hvy)^{-T} \hat\nabla\hat\varphi(\hvy)^{\,T}.
  \end{align}
  Let $\hvn$ be a (unit) normal field on $\hat\Sigma$ and let $\hvt_1$,
  $\hvt_2$ be two orthonormal (unit) tangential fields. Then,
  \begin{align*}
    \vn = \hat F(\hvy)^{-T}\hvn /\|\hat F(\hvy)^{-T}\hvn\|_{\ell^2},
    \quad
    \vt_i = \hat F(\hvy)\hvt_i /\|\hat F(\hvy)\hvt_i\|_{\ell^2},
    \; i=1,2.
  \end{align*}
  are a (unit) normal field and orthonormal tangential fields on
  $\Sigma$. Moreover,
  \begin{align}
    \dy\;=\;\hat J\dhy,
    \qquad \qquad
    \doy\;=\;\|\hat F(\hvy)^{-T}\hvn\|_{\ell^2}\hat J\dhohy.
    \label{eq:transformed_volume_surface}
  \end{align}
\end{lemma}
\begin{proof}
  We refer to \cite[§\,2.1.2]{richter2017} for a detailed
  discussion of the transformation identities. The transformation formula
  for $\vn$ and $\doy$ is a direct consequence of \emph{Nanson's formula}.
\end{proof}
\begin{lemma}
  \label{lem:tangential_transformation}
  In the setting of Lemma~\ref{lem:transformation} we have the identity
  \begin{align*}
    \tnabla\varphi
    \;=\;
    \sum_k \vt_k\vt_k^T(\hat F(\hvy)^{-T}\hat\nabla\hat\varphi)
    \;=\;
    \sum\limits_k \frac{\hat F(\hvy)\hvt_k}{\|\hat F(\hvy)\hvt_k\|^2_{\ell^2}} \,
    \hvt_k\cdot\hat\nabla\hat\varphi.
  \end{align*}
\end{lemma}
\begin{proof}
  The statement is a direct consequence of Lemma~\ref{lem:transformation}
  and the fact that
  \begin{align*}
    \tnabla\varphi = \sum_k \vt_k \vt_k^T \nabla \varphi.
  \end{align*}
\end{proof}
We are now in a position to recast \eqref{eq:eff} and
\eqref{eq:cell_problem} in reference coordinates. We introduce
\begin{equation*}
  \vec H(\hat Y, \hat\Sigma):=
  \big\{\hvvarphi\in \vec H^1_{\text{per}}(\hat Y,\C)\;:\;
  \;\hvt_k\cdot \hat\nabla\hvvarphi\in L^2(\hat \Sigma,\C^{d\times
  d}),\,k=1,2\big\}
\end{equation*}
The corrector $\hvchi=\big(\hat{\chi}_i\big)_{i=1}^3 \in \vec H(\hat
Y, \hat \Sigma)$ is determined by
\begin{equation}
  \label{eq:cell_problem_ale}
  E(\hvchi,\hvvarphi;\hvq) \;=\;0
  \quad\text{for all }\hvvarphi\in\vec H(\hat Y, \hat \Sigma),
\end{equation}
where,
\begin{multline}
  \label{eq:cell_problem_ale_bilinear}
  E(\hvchi,\hvvarphi;\hvq) :=
  \int_{\hat Y} \eps(\vx,\vy(\hvy))
    (\unit + \hat F(\hvy)^{-T}\hat\nabla\hvchi^{\:T})
    \cdot(\hat F(\hvy)^{-T}\overline{\hat\nabla\hvvarphi^{\:T}})\hat J\dhy\\
    - \frac{1}{i\omega} \int_{\hat \Sigma} \sigma(\vx,\vy(\hvy))
    \Big(\sum_k \vt_k\vt_k^T\big(\unit+\hat
    F(\hvy)^{-T}\hat\nabla\hvchi^{\:T}\big)\Big)
    \\
    \cdot\Big(
    \sum_k \vt_k\vt_k^T\hat
    F(\hvy)^{-T}\overline{\hat\nabla\hvvarphi^{\:T}}\Big)
    \|\hat F(\hvy)^{-T}\hvn\|_{\ell^2}\hat J \dhohy.
\end{multline}
Here, $\unit$ is the unit matrix. We note that $\sigma(\vx,\vy(\hvy))$
and $\eps(\vx,\vy(\hvy))$, still denote tensors acting on gradients in
transformed (non-reference) coordinates. Similarly, the effective
permittivity tensor is given by
\begin{multline}
  \label{eq:eff_ale}
  \eff_{ij}(\hvchi;\hvq) = \int_{\hat Y} \eps(\vx,\vy(\hvy))
  (\vec e_j+\hat F(\hvy)^{-T}\hat\nabla\hat\chi_j^{\:T})\cdot(\vec
  e_i+\hat F(\hvy)^{-T}\overline{\hat\nabla\hat\chi_i^{\:T}})\hat J\dhy
  \\
  - \frac{1}{i\omega} \int_{\hat\Sigma} \sigma(\vx,\vy(\hvy))
  \Big(\tve{j} + \sum_k \vt_k\vt_k^T\hat
  F(\hvy)^{-T}\hat\nabla\hvchi_j^{\:T}\Big)
  \\
  \cdot
  \Big(\tve{i} + \sum_k \vt_k\vt_k^T\hat
  F(\hvy)^{-T}\overline{\hat\nabla\hvchi_i^{\:T}}\Big)
  \|\hat F(\hvy)^{-T}\hvn\|_{\ell^2}\hat J \dhohy.
\end{multline}

\subsection{Well-posedness, regularity and dependence on deformation field}
\label{subsec:well_pos}
We now examine the well-posedness and regularity of problem
\eqref{eq:cell_problem_ale}. To this end we introduce two transformed
tensors:
\begin{gather}
  \label{eq:transformed_tensors}
  \hat \eps(\vx,\hvy) := \hat F(\hvy)^{-1} \eps(\vx,\vy(\hvy)) \hat
  F(\hvy)^{-T}\hat J,
  \quad
  \hat{\sigma}(\vx,\hvy) := \sum_{mn}
  \hat{\sigma}_{mn}\,\hvt_m\hvt_n^T,
  \\\notag
  \text{where}\quad
  \hat{\sigma}_{mn} := \frac{\sigma_{mn}(\vx,\vy(\hvy))\|\hat
  F(\hvy)^{-T}\hvn\|_{\ell^2}\hat J} {\|\hat F(\hvy)\hvt_m\|^2_{\ell^2}\|
  \|\hat F(\hvy)\hvt_n\|^2_{\ell^2}},
\end{gather}
and $\sigma_{mn}$ is defined by ${\sigma}(\vx,\vy(\hvy)) =: \sum_{mn}
{\sigma}_{mn}\,\vt_m\vt_n^T.$ We make the following observations:

\begin{lemma}
  \label{lem:transformed_tensors_are_nice}
  Assume that $\hvq(\hvy)\in \vec D(\hat Y, \hat\Sigma)$
  and $0 < \delta \le \hat J(\hvy)$ is uniformly bounded. Then, the
  tensors $\hat\eps(\hvy)$ and $\hat\sigma(\hvy)$ are bounded, complex and
  tensor-valued functions and $\Re(\hat\eps(\hvy))$ and
  $\text{Im}{(\hat\sigma(\hvy))}$ are symmetric,
  $\text{Im}(\hat\eps(\hvy))$ and $\text{Re}{(\hat\sigma(\hvy))}$ are
  symmetric and uniformly elliptic with a constant depending on $\delta$.
\end{lemma}
\begin{lemma}
  \label{lem:equivalence}
  The bilinear term \eqref{eq:cell_problem_ale_bilinear} can be
  equivalently written as follows:
  \begin{multline}
    \label{eq:cell_problem_ale_ng}
    E(\hvchi,\hvvarphi;\hvq) :=
    \int_{\hat Y} \hat\eps(\vx,\hvy)
    \big(\hat F(\hvy)^T + \hat\nabla\hvchi^{\:T}\big)
    \cdot\overline{\hat\nabla{\hvvarphi}^{\:T}}\dhy
    \\
    - \frac{1}{i\omega} \int_{\hat \Sigma} \hat\sigma(\vx,\hvy)
    \big(\htangent{(\hat F(\hvy)^T)} +
    \htnabla \hvchi^{\:T}\big)\cdot\overline{\htnabla\hvvarphi^{\:T}} \dhohy.
  \end{multline}
  Similarly, \eqref{eq:eff_ale} takes the form:
  \begin{multline}
    \label{eq:eff_ale_ng}
    \eff_{ij}(\hvchi;\hvq) =
    \int_{\hat Y} \hat\eps(\vx,\hvy)
    \big(\hat F(\hvy)^T\vec e_j + \hat\nabla\hvchi_j^{\:T}\big)
    \cdot\big(\hat F(\hvy)^T\vec e_i +
    \overline{\hat\nabla\hvchi_i^{\:T}}\big)
    \dhy
    \\
    - \frac{1}{i\omega} \int_{\hat \Sigma} \hat\sigma(\vx,\hvy)
    \big(\htangent{(\hat F(\hvy)^T\vec e_j)} + \htnabla
    \hvchi_j^{\:T}\big)
    \cdot
    \big(\htangent{(\hat F(\hvy)^T\vec e_i)} +
    \overline{\htnabla \hvchi_i^{\:T}}\big) \dhohy.
  \end{multline}
\end{lemma}

Detailed proofs of Lemmas \ref{lem:transformed_tensors_are_nice} and
\ref{lem:equivalence} are given in Appendix \ref{app:estimates}.
Lemma~\ref{lem:equivalence} implies, in particular, that the deformed cell
problem \eqref{eq:cell_problem_ale} has the same structure as
\eqref{eq:cell_problem} with a slightly modified forcing. We can thus
summarize:

\begin{theorem}
  \label{thm:well_posedness_ale}
  Under the assumptions on the deformation field $\hvq(\hvy)$ as stated in
  Lemma~\ref{lem:transformed_tensors_are_nice} problem
  \eqref{eq:cell_problem_ale} is well-posed.
\end{theorem}

\begin{proof}
  Lemma~\ref{lem:tangential_transformation} ensures that equation
  \eqref{eq:cell_problem_ale} is the same as \eqref{eq:cell_problem} with
  modified material tensors and a modified forcing.
  Lemma~\ref{lem:transformed_tensors_are_nice} ensures that all assumptions
  on the material tensors stated in Theorem~\ref{thm:well_posedness} hold
  true. Well-posedness is thus an immediate consequence of
  Theorem~\ref{thm:well_posedness}.
\end{proof}

We finish the discussion by introducing a number of robustness and
regularity results that will be used later to justify the optimization
approach.
\begin{theorem}
  \label{thm:robustness}
  For the unique solution $\hvchi \in \vec H(\hat Y, \hat\Sigma)$ to
  \eqref{eq:cell_problem_ale}, we have the following a priori estimate:
  \begin{multline*}
    \|\hat\nabla \hvchi\|^2_{L^2(\hat Y)}
    +\frac1\omega\|\htnabla \hvchi\|^2_{L^2(\hat\Sigma)}
    \\
    \le C\, \left(
    \|\hat\eps(\vx,\hvy)\|^2_{L^\infty(\hat Y)}
    \|\hat F^T\|^2_{L^2(\hat Y)}
    + \frac{1}{\omega}
    \|\hat \sigma(\vx,\hvy)\|^2_{L^\infty(\hat\Sigma)}
    \|\hat F^T\|^2_{L^2(\hat\Sigma)}
    \right)
  \end{multline*}
  for a constant $C>0$ only depending on $\hat\Sigma$ and the lower bound
  $\delta$ of $\hat J(\hat \vx)$ as defined in
  Lemma~\ref{lem:transformed_tensors_are_nice}.
\end{theorem}
\begin{proof}
  The statement is a consequence of
  Lemmas~\ref{lem:transformed_tensors_are_nice} and~\ref{lem:equivalence}.
  We start by testing \eqref{eq:cell_problem_ale_ng} with
  $\hvvarphi=\overline{\hvchi(\hvq)}$ and taking the imaginary
  part. Recalling that $\Re\hat\eps$, $\Im\hat\eps$, $\Re\hat\sigma$,
  $\Im\hat\sigma$ are symmetric by virtue of
  Lemma~\ref{lem:transformed_tensors_are_nice} we arrive at
  \begin{multline*}
    \int_{\hat Y} \Im\hat\eps\,
    \hat\nabla\hvchi^{\:T}
    \cdot\overline{\hat\nabla{\hvchi}^{\:T}}\dhy
    + \frac{1}{\omega} \int_{\hat \Sigma} \Re\hat\sigma\,
    \htnabla \hvchi^{\:T}\cdot\overline{\htnabla\hvchi^{\:T}} \dhohy
    \\
    =\;
    -\Im\Big\{\int_{\hat Y} \hat\eps\, \big(\hat F(\hvy)^T\big)
    \cdot\overline{\hat\nabla{\hvchi}^{\:T}}\dhy
    + \frac{1}{\im\omega} \int_{\hat \Sigma} \hat\sigma\,
    \htangent{\big(\hat F(\hvy)^T \big)}
  \cdot\overline{\htnabla\hvchi^{\:T}} \dhohy\Big\}.
  \end{multline*}
  The statement now follows from using Young's inequality for both terms on
  the right side and uniform ellipticity of $\Im\hat\eps$ and
  $\Re\hat\sigma$ with a $\delta$-dependent constant that was
  established in Lemma~\ref{lem:transformed_tensors_are_nice}.
\end{proof}

\begin{theorem}
    \label{thm:dependence}
    Under the assumptions on the deformation field $\hvq(\hvy)$ as stated
    in Lemma~\ref{lem:transformed_tensors_are_nice}, the corrector
    $\hvchi(\hvq)$ given by the cell problem $\eqref{eq:cell_problem_ale}$
    depends at least Lipschitz-continuously on $\hvq$. More precisely, let
    $\hvq_1, \hvq_2 \in \vec D(\hat Y, \hat\Sigma)$ be arbitrarily
    chosen such that the assumptions of
    Lemma~\ref{lem:transformed_tensors_are_nice} are satisfied. Let
    $\hvchi_1, \hvchi_2 \in \vec H(\hat Y, \hat\Sigma)$ denote the
    solutions to \eqref{eq:cell_problem_ale} for deformation fields
    $\hvq_1$ and $\hvq_2$, respectively. Then, assuming that $\hvq_1$ is
    suitably close to $\hvq_2$, we have:
    \begin{multline}
        \|\hat\nabla (\hvchi_1 - \hvchi_2)\|_{L^2(\hat Y)}^2 +
        \frac{1}{\omega}
        \|\htnabla (\hvchi_1 - \hvchi_2)\|_{L^2(\hat \Sigma)}^2
        \\
        \le\;
        C(\hvq)\,\Big\{
        \|\hat\eps_1(\vx,\hvy) - \hat\eps_2(\vx,\hvy)\|_{L^{\infty}(\hat
        Y)}^2
        +
        \frac1\omega\, \|\hat\sigma_1(\vx,\hvy) -
        \hat\sigma_2(\vx,\hvy)\|_{L^{\infty}(\hat\Sigma)}^2
        \Big\}
        \\
        \le\;C(\hvq_2)\,\Big\{\|\hvq_1 - \hvq_2\|^2_{L^\infty(\hat Y)} +
        \|\hvq_1 - \hvq_2\|^2_{L^\infty(\hat \Sigma)}\Big\},
    \end{multline}
    where the constant $C$ only depends on $\hat F_2(\hvy)$ and thus on
    $\hvq_2$.
\end{theorem}
A detailed proof of Theorem~\ref{thm:dependence} is given in
Appendix~\ref{app:estimates}.
\begin{remark}
  The Lipschitz continuity of $\hvq \mapsto \hvchi(\hvq)$ implies
  that this mapping is already G\^ateaux differentiable almost
  everywhere, see, e.g.,~\cite{BogachevMayerWolf:1996}. This justifies to
  take Gâteaux derivatives with respect to the control $\hvq$ in the adjoint
  formulation; see Section~\ref{sec:optimization}.
\end{remark}
\begin{theorem}
  \label{thm:regularity}
  Let $\hat \Sigma$ be a smooth, closed hypersurface, i.\,e.,
  $\partial\hat\Sigma = 0$. Suppose there exist a smooth, $\hat Y$-periodic
  extensions $\hat{\vec\tau}_i(\vx):\hat Y \to\R^n$ of the tangential
  fields $\hat{\vec\tau}_i$ with $|\hat{\vec\tau}_i(\vx)|\le 1$ for all
  $\vx\in\hat Y$.
  Then, under the assumptions stated in Theorem~\ref{thm:robustness} and
  provided that $\hat\varepsilon(\vx,\hvy)$ and $\hat\sigma(\vx,\hvy)$ are
  sufficiently regular the following stability estimate holds true:
  \begin{multline}
    \label{eq:regularity}
    \|\hat\nabla\htnabla\hvchi\|^2_{L^2(\hat Y)}
    +\frac1\omega\|\htnabla\htnabla\hvchi\|^2_{L^2(\hat\Sigma)}
    \\
    \le \;C\;
    \max\big\{
    \|\hat\eps(\vx,\hvy)\|^2_{W^{1,\infty}(\hat Y)}
    , \frac{1}{\omega}
    \|\hat \sigma(\vx,\hvy)\|^2_{W^{1,\infty}(\hat\Sigma)}\Big\}
    \\
    \times
    \Big\{
    \|\hat\eps(\vx,\hvy)\|^2_{W^{1,\infty}(\hat Y)}
    \|\hat F^T \vec e_i\|^2_{H^1(\hat Y)}
    + \frac{1}{\omega}
    \|\hat \sigma(\vx,\hvy)\|^2_{W^{1,\infty}(\hat\Sigma)}
    \|\hat F^T \vec e_i\|^2_{H^1(\hat\Sigma)}\Big\}.
  \end{multline}
  Here, the constant $C$ only depends on $\hat\Sigma$ and the chosen
  extension of $\hat{\vec\tau}$.
\end{theorem}
A detailed proof of Theorem~\ref{thm:regularity} is given in
Appendix~\ref{app:regularity}.


\section{Shape optimization problem and adjoint formulation}
\label{sec:optimization}
The previous section establishes that the effective permittivity
tensor $\eff(\hvchi;\hvq)$ given by \eqref{eq:eff_ale_ng} enjoys a
sufficient degree of regular dependence on $\hvq$ to formulate an
optimization problem; and solve it by means of derivative based optimization
methods. To this end, we introduce a cost functional with the
target to minimize the Frobenius norm between $\eff(\hvchi,\hvq)$ and a
given target permittivity tensor. A particular difficulty that has to be
addressed is the necessity to maintain a lower bound, $\hat{J} \ge \delta >
0$, on the transformation determinant. We thus set:
\begin{definition}
  For a given target tensor $\eps^{\text{target}}$, introduce a cost
  function
  \begin{gather}
    \label{eq:cost_functional}
    C(\hvchi;\hvq) \;:=\; \frac12\,\big\|\eff(\hvchi,\hvq)  -
    \eps^{\text{tgt.}}\big\|_{\text{Fr.}}^2 \;+
    \frac\alpha2\, \|\nabla\hvq\|^2_{w} +\;
    \beta\int_{\hat Y}P(\hvy;\hvq)\dhy,
    \\
    \|\nabla\hvq\|^2_{w}\;:=\;
    \int_Y\,w(\vec{\hat y})|\nabla\hvq|^2\dhy,
    \quad
    P(\hvy;\hvq)\;:=\;
    \begin{cases}
      \frac12\,\frac{(\hat J(\hvy) -1)^2}{|\hat J(\hvy)| +
      \hat J(\hvy)}&
      \text{if } \hat J(\hvy) < 1,
      \\[0.6em]
      \frac12\,{(\hat J(\hvy)-1)^2} &
      \text{if } \hat J(\hvy) \ge 1.
    \end{cases}
    \notag
  \end{gather}
  Here, $\alpha>0$ is an appropriately chosen Tikhonov regularization
  parameter, and the coefficients $\beta>0$ control a penalty on the
  deviation of transformation determinant from $\hat J = 1$.
  Moreover, $w(\vec{\hat y})>0$ is a weight function that will be
  chosen later.
\end{definition}
The penalty term $P(\hvy;\hvq)$ is chosen to  provide a barrier that
enforces positivity of the transformation gradient $\hat{J}$ and penalizes
a deviation away from 1. Strictly speaking, this penalty only enforces
positivity of the transformation gradient, $\hat{J} > 0$, but not a
uniform lower bound. This is not a problem for the discretized setting
of our numerical computations (see
Section~\ref{sec:numerical_examples}) because the finite dimensionality
will ensure that $\hat{J}$ remains bounded away from 0, though with a possibly discretization
dependent constant $\delta_h$. Nevertheless, if a guaranteed lower bound $\delta$ is
desired then $P(\hvy;\hvq)$ can be easily modified to accommodate this by
substituting $\hat{J}$ by $\hat{J}-\delta$.

We now seek solutions $(\hvchi,\hvq)\in X:=\vec H(\hat Y, \hat \Sigma)\times
\vec D(\hat Y, \hat\Sigma)$ of the optimization problem
\begin{align}
  \label{eq:optimization_problem}
  \min_{(\hvchi, \hvq)\in X} C(\hvchi;\hvq) \quad\text{subject
  to}\quad E(\hvchi,\hvvarphi;\hvq)\,=\,0\;\;\text{for all}\;\;
  \hvvarphi\in\vec H(\hat Y, \hat \Sigma).
\end{align}

\subsection{Adjoint formulation}
\label{subsec:adjoint_approach}
In order to derive the optimality condition
\eqref{eq:optimization_problem}, we would have to compute a partial
derivative in multiple directions, which is inconvenient. Therefore, we use
an adjoint formulation, see, e.g.,~\cite{BeckerMeidnerVexler:2007,
HinzePinnauUlbrichUlbrich:2009, Troeltzsch:2010}.
\begin{definition}
  Define a Lagrangian
  \begin{gather*}
    \mathcal{L} \;:\; \vec H(\hat Y, \hat\Sigma) \times
    \vec H(\hat Y, \hat\Sigma) \times \vec D(\hat Y, \hat\Sigma) \to \C,
    \\[0.25em]
    \mathcal{L}(\hvchi,\hz;\hvq) = C(\hvchi;\hvq) -  E(\hvchi,\hz;\hvq),
  \end{gather*}
  where we have introduced a Lagrange multiplier $\hz$ for the PDE
  constraint in \eqref{eq:optimization_problem}.
  For a given deformation field $\hvq\in H^1_0(\hat Y,\mathbb{C})^3$ we
  further introduce a \textit{state equation}:
  \begin{gather}
    \label{eq:state_equation}
    \text{find }\hvchi\in \vec H(\hat Y, \hat\Sigma)
    \text{ s.\,t. }
    \mathcal{L}'_{\hz}(\hvchi,\hz;\hvq)[\delta\hz] =
    0\;\;\forall\delta\hz\in\vec H(\hat Y, \hat\Sigma)
  \end{gather}
  and denote the solution of the state equation by $\hvchi(\hvq)$. Here,
  $\mathcal{F}'_{\hz}[\delta\hz]$ denotes the Gâteaux derivative of a
  functional $F$ with respect to $\hz$ in direction $\delta\hz$. We then
  introduce an adjoint equation:
  \begin{gather}
    \label{eq:adjoint_equation}
    \text{find }\hz\in \vec H(\hat Y, \hat\Sigma)\text{ s.\,t. }
    \mathcal{L}'_{\hvchi}(\hvchi(\hvq),\hz;\hvq)[\delta\hvchi] = 0
    \;\;\forall\delta\hvchi\in\vec H(\hat Y, \hat\Sigma).
  \end{gather}
\end{definition}
The central observation is the fact that a solution to
\eqref{eq:optimization_problem} is a critical point of
$\mathcal{L}(\hvchi,\hz;\hvq)$; see \cite{BeckerMeidnerVexler:2007,
HinzePinnauUlbrichUlbrich:2009, Troeltzsch:2010}. For the sake of
completeness we summarize:
\begin{lemma}[First order necessary conditions~\cite{BeckerMeidnerVexler:2007, HinzePinnauUlbrichUlbrich:2009, Troeltzsch:2010}]
  \label{lem:adjoint_formulation}
  The solution $(\hvchi^\ast,\hvq^\ast)$ of \eqref{eq:optimization_problem}
  coincides with a critical point $(\hvchi^\ast,\hz^\ast,\hvq^\ast)$ of the
  Lagrangian $\mathcal{L}(\hvchi,\hz;\hvq)$.
\end{lemma}

\begin{proof}
  Let $(\hvchi^\ast,\hvq^\ast)$ be a solution to
  \eqref{eq:optimization_problem} and let $\hz^\ast$ be the solution to the
  adjoint equation \eqref{eq:adjoint_equation}. We then have
  $\mathcal{L}'_{\hz}(\hvchi^\ast,\hz^\ast;\hvq^\ast)[\delta\hz]=0$ and
  $\mathcal{L}'_{\hvchi}(\hvchi^\ast,\hz^\ast;\hvq^\ast)[\delta\hvchi]=0$
  by virtue of \eqref{eq:state_equation} and \eqref{eq:adjoint_equation}.

  For an arbitrary deformation $\hvq\in\vec D(\hat Y,\hat\Sigma)$ let
  $\hvchi(\hvq)$ denote the unique solution to the state equation
  \eqref{eq:state_equation}, as well as $\hz(\hvq)$ denote the unique
  solution to the adjoint equation \eqref{eq:adjoint_equation}. We now
  introduce the functional $c(\hvq):=C(\hvchi(\hvq);\hvq)=
  \mathcal{L}(\hvchi(\hvq),\hz(\hvq);\hvq)$ and make the observation that
  $\hvq$, by virtue of being an optimum, is necessarily a critical point of
  $c(\hvq)$, i.\,e., $c'_{\hvq}(\hvq)[\delta\hvq] = 0$ for all
  $\delta\hvq\in \vec D(\hat Y,\hat\Sigma)$. Using above identity and the
  chain rule we get for all $\delta\hvq\in \vec D(\hat Y,\hat\Sigma)$:
  \begin{align}
    \label{eq:central_observation}
    0 \;&=\; c'(\hvq)[\delta\hvq]
    \\\notag
    &=\; \mathcal{L}'_{\hvchi}(\hvchi(\hvq),\hz(\hvq);\hvq)
    \big[\hvchi'_{\hvq}(\hvq)[\delta\hvq]\big]
    + \mathcal{L}'_{\hz}(\hvchi(\hvq),\hz(\hvq);\hvq)[\hz'(\hvq)[\delta\hvq]]
    \\\notag
    &\quad\;
    + \mathcal{L}'_{\hvq}(\hvchi(\hvq),\hz(\hvq);\hvq)[\delta \hvq]
    \\\notag
    &=\; \mathcal{L}'_{\hvq}(\hvchi(\hvq),\hz(\hvq);\hvq)[\delta \hvq],
  \end{align}
  where for the last equality we have exploited the fact that the first two
  terms vanish due to $\hvchi(\hvq)$ and $\hz(\hvq)$ solving the state and
  adjoint equations, respectively. $(\hvchi^\ast,\hz^\ast,\hvq^\ast)$ is
  thus a critical point of the Lagrangian $\mathcal{L}(\hvchi,\hz;\hvq)$.
\end{proof}%
The previous lemma is based on the fact the Lagrangian $\mathcal{L}$ can be
used to describe the derivative of
$c(\hvq)$~\cite{BeckerMeidnerVexler:2007, HinzePinnauUlbrichUlbrich:2009,
Troeltzsch:2010}; for all $\delta\hvq\in \vec D(\hat Y,\hat\Sigma)$:
\begin{align*}
  c'(\hvq)[\delta\hvq] \;=\;
  \mathcal{L}'_{\hvq}(\hvchi(\hvq),\hz(\hvq);\hvq)[\delta \hvq].
\end{align*}
We will make use of this result in the numerical algorithm to compute a
gradient direction for the deformation. Noting that $c'(\hvq)$ is an
element of the dual space of $H^1_0(\hat Y,\mathbb{C})^{3}$, we find the
Riesz-representation, or gradient, as follows:
\begin{definition}[Gradient equation]
  Given a deformation field $\hvq$ and corresponding $\hz(\hvq)$,
  $\hvchi(\hvq)$ given by \eqref{eq:state_equation} and
  \eqref{eq:adjoint_equation}, we find $\vec{\delta c}(\hvq)\in \vec D(\hat
  Y,\hat\Sigma)$ by solving the \emph{gradient equation}
  \begin{align}
    \label{eq:gradient_equation}
    \int_{\hat Y}
    \nabla\vec{\delta c}(\hvq)\cdot\nabla\delta\hvq\dhy
    \;=\;
    \mathcal{L}'_{\hvq}(\hvchi(\hvq),\hz(\hvq);\hvq)[\delta \hvq]
    \qquad \forall \delta\hvq \in \vec D(\hat Y, \hat\Sigma).
  \end{align}
\end{definition}

\subsection{Finite element discretization and optimization framework}
\begin{algorithm}[bt]
  Given $\hvq\in \vec D(\hat Y,\hat\Sigma)$
  \vspace{0.25em}
  \begin{itemize}
    \item[a)]
      compute a solution $\hvchi(\hvq) \in \vec H(\hat Y,\hat\Sigma)$ of the
      state equation \eqref{eq:state_equation},
      \begin{equation*}
        \mathcal{L}'_{\hz}(\hvchi,\hz;\hvq)[\delta \hz] =
        - E(\hvchi,\delta \hz;\hvq) = 0\;\;\forall
        \delta\hz\in\vec H(\hat Y,\hat\Sigma);
      \end{equation*}
    \item[b)]
      compute a solution $\hz(\hvq)$ of the adjoint equation
      \eqref{eq:adjoint_equation},
      \begin{equation*}
        \mathcal{L}'_{\hvchi}(\hvchi(\hvq),\hz;\hvq)[\delta
        \hvchi] = 0 \;\;\forall
        \delta\hvchi\in\vec H(\hat Y,\hat\Sigma);
      \end{equation*}
    \item[c)]
      solve the gradient equation \eqref{eq:gradient_equation},
      \begin{equation*}
        \int_{\hat Y}
        \nabla\vec{\delta c}(\hvq)\cdot\nabla\delta\hvq\dhy
        \;=\;
        \mathcal{L}'_{h,\hvq}(\hvchi(\hvq),\hz(\hvq);\hvq)[\delta
        \hvq] \qquad \forall \delta\hvq \in
        \vec D(\hat Y,\hat\Sigma).
      \end{equation*}
  \end{itemize}%
  Return $\vec{\delta c}(\hvq)$.
  \caption{Computing the $H^1$-gradient $\vec{\delta c}(\hvq)$ of the cost
    functional $c(\hvq)$ by means of the adjoint formulation; see
    Lemma~\ref{lem:adjoint_formulation}.}
  \label{alg:compute_gradient}
\end{algorithm}
\begin{algorithm}[tb]
  Given an initial guess
  $\hvq_h^0 \in \vec D_h$ and an initial
  approximate inverse Hessian operator
  $B^0\in \mathcal{L}(\vec D_h,\vec D_h)$ iterate:
  \begin{itemize}
    \item[a)]
      Compute the gradient $\vec{\delta c}_h(\hvq_h^n) \in \vec D_h$
      with Algorithm~\ref{alg:compute_gradient}, and obtain a search
      direction $\vec p^n \in \vec D_h$ by setting
      \begin{align*}
        \vec p^n = - B^n\vec{\delta c}_h(\hvq_h^n).
      \end{align*}
    \item[b)]
      Perform an Armijo backtracking line search with parameters $\beta \in
      (0,1)$ and $\gamma \in (0,1/2)$ to find the maximal $\lambda^n \in
      \{1,\beta, \beta^2,\ldots\}$ satisfying the \emph{Armijo condition}
      $\vec{c}_h(\hvq_h^n+\lambda^n\vec p^n) \le \vec{c}_h(\hvq_h^n)
      +\gamma \lambda^n (\nabla \vec{\delta c}_h(\hvq_h^n),\nabla p^n).$
      Then, update
      \begin{equation*}
        \vec s^n := \lambda^n \vec p^n,\quad
        \hvq_h^{n+1} = \hvq_h^n + \vec s^n.
      \end{equation*}
    \item[c)]
      Update the approximate inverse Hessian matrix using a damped inverse
      BFGS update formula derived in~\cite{HerterWollner:2023} to assert
      positive definiteness of the operator $B^{n+1}$. To this end, compute
      \begin{equation*}
        \vec{y}^n = \vec{\delta c}_h(\hvq_h^{n+1}) -\vec{\delta
        c}_h(\hvq_h^n).
      \end{equation*}
      Define a scaling parameter
      \begin{equation*}
        \theta_n =
        \begin{cases}
          1 & (\nabla \vec{y}^n,\nabla \vec{s}^n)\ge 0.2 (\nabla
          \vec{y}^n,\nabla B^n \vec{y}^n),
          \\[0.25em]
          0.8 \frac{ (\nabla \vec{y}^n,\nabla B^n \vec{y}^n)}
          {(\nabla \vec{y}^n,\nabla B^n\vec{y}^n)-(\nabla \vec{y}^n,\nabla
          \vec{s}^n) } & \text{otherwise},
        \end{cases}
      \end{equation*}
      and with $\widehat{\vec{s}}^n = \theta_n \vec{s}^n + (1-\theta_n) B^n
      \vec{y}^n$ set
      \begin{multline*}
        B^{n+1} \;=\; B^n + \frac{(\widehat{\vec{s}}^n-B^n\vec{y}^n)(\nabla
        \widehat{\vec{s}}^n,\nabla \cdot) + \widehat{\vec{s}}^n (\nabla
        (\widehat{\vec{s}}^n-B^n\vec{y}^n),\nabla \cdot)}{(\nabla
        \vec{y}^n,\nabla \widehat{\vec{s}}^n)}
        \\[0.25em]
        - \frac{(\nabla (\widehat{\vec{s}}^n-B^n\vec{y}^n),\nabla
        \vec{y}^n)}{(\nabla \vec{y}^n,\nabla \widehat{\vec{s}}^n)^2}
        \widehat{\vec{s}}^n (\nabla \widehat{\vec{s}}^n,\nabla \cdot),
      \end{multline*}
      where, of course, the matrix corresponding to the operator
      $B^{n+1}$ is never constructed directly. Instead, the application of
      $B^{n+1}$ to the direction $\vec{\delta c}_h(\hvq_h^n)$
      is computed by storing the vectors $\widehat{\vec{s}}^n$ and
      $B^n\vec{y}^n$.
    \item[d)]
      If the stopping criterion was reached, return $\hvq_h^{n+1}$,
      otherwise continue at (a).
  \end{itemize}
  \caption{The inverse damped BFGS algorithm for finding an approximate solution
    $\hvq_h^\ast$ of the optimization problem
    \eqref{eq:optimization_problem}.}
  \label{alg:bfgs}
\end{algorithm}
For our numerical tests we use the optimization toolkit
\texttt{DOpElib}~\cite{dopelib} which is based on the finite element
library \texttt{deal.II}~\cite{dealIIcanonical,dealII94}. The library
supports a variety of finite element formulations based on quadrilateral
(in 2d) and hexahedrical (in 3d) meshes.

Let $\mathcal{T}_h$ be a partition of $\hat Y$ into shape-regular
(quadrilateral or) hexahedral elements that are \emph{fitted} to the
hypersurface $\hat\Sigma$. This is to say, we make the assumption that
every element that is intersected by $\hat\Sigma$ has a face for which all
four vertices of the face are located on $\hat\Sigma$. We denote by
$\hat\Sigma_h$ the set of all faces for which all vertices of the face are
located on $\hat\Sigma$. By slight abuse of notation we will interpret
$\hat\Sigma_h$ either as a collection of faces or as the polyhedral
hypersurface created by the union of all faces. We denote by
$\big\{\phi_i^h\big\}^{\mathcal{N}}_{i=1}$ the Lagrange basis of
$\mathbb{Q}_1(\mathcal{T}_h)$, the space of piecewise (bilinear) trilinear
finite elements defined on $\mathcal{T}_h$. Note that the tangential
derivative $\htnabla\varphi$ on a face $f\in\hat\Sigma_h$ of a finite
element function $\varphi\in\mathbb{Q}_1(\mathcal{T}_h)$ is single valued.
We can thus introduce a discrete bilinear form corresponding to
\eqref{eq:cell_problem_ale_ng}
\begin{multline}
  \label{eq:discrete_cell_problemm_ale}
  E_h(\hvchi_h,\hvvarphi_h;\hvq_h) :=
  \int_{\hat Y} \hat\eps_h(\vx,\hvy)
  \big(\hat F_h(\hvy)^T + \hat\nabla\hvchi_{h}^{\:T}\big)
  \cdot\hat\nabla{\hvvarphi}_h^{\:T}\dhy
  \\
  - \frac{1}{i\omega} \int_{\hat \Sigma_h} \hat\sigma_h(\vx,\hvy)
  \big(\htangent{(\hat F_h(\hvy)^T)} +
  \htnabla \hvchi_h^{\:T}\big)\cdot\htnabla\hvvarphi_h^{\:T} \dhohy,
\end{multline}
for $\hvchi_h, \hvvarphi_h \in \vec
H_h:=\mathbb{Q}_1(\mathcal{T})^{2\times3}$ and $\hvq_h \in \vec D_h :=
\mathbb{Q}_1(\mathcal{T})^{3}.$
Here, $\hat\eps_h$, $\hat\sigma_h$ and $\hat F_h$ are computed with respect
to the discrete objects $\hvq_h$ and $\hat\Sigma_h$. Similarly, we
introduce a discrete counterpart $\eff_h$ of the effective permittivity
tensor given by \eqref{eq:eff_ale_ng} and we set
\begin{align}
  \label{eq:discrete_cost_functional}
  C_h(\hvchi_h;\hvq_h) := \frac12\,\big\|\eff_h(\hvchi_h;\hvq_h)  -
  \eps^{\text{trgt}}\big\|_{\text{Fr.}}^2
  +\frac\alpha2\, \|\nabla\hvq_h\|^2_{w_h}
  +\beta\,\int_{\hat Y}P(\hvy;\hvq_h)\dhy.
\end{align}
Here, we choose the following weight function in order to penalize
more the deformation gradient on mesh cells at the interface that ensures
that the discrete interface $\Sigma_h$ retains a sufficient degree of
smoothness:
\begin{equation*}
  w_h(\vec{\hat y}) =
  \begin{cases}
    \begin{aligned}
      &1+\alpha_{\Sigma}/\text{diam}{K} &\quad &\text{for } \vec{\hat y} \in K
      \text{ with }\partial K\cap\delta\Sigma_h\not=\emptyset,
      \\
      &1 &\quad &\text{otherwise.}
    \end{aligned}
  \end{cases}
\end{equation*}

By introducing a discrete Lagrangian
\begin{gather*}
  \mathcal{L}_h \;:\; \vec H_h \times \vec H_h \times \vec D_h \to \C,
  \\[0.25em]
  \mathcal{L}(\hvchi_h,\hz_h;\hvq_h) = C_h(\hvchi;\hvq) -
  E_h(\hvchi,\hz;\hvq),
\end{gather*}
we are now in a position to use the adjoint formulation to compute a
discrete gradient~\cite{dopelib}; see Algorithm~\ref{alg:compute_gradient}.
We use an inverse BFGS algorithm with damping as proposed in~\cite{HerterWollner:2023}
to find an approximate minimum of $C_h(\hvq_h)$; see
Algorithm~\ref{alg:bfgs}. The $\lambda^n$ in the above algorithm is the
step size it is chosen by an Armijo backtracking linesearch and changes at
every iteration. As mentioned, $B^n$ is an approximate inverse Hessian
matrix. The update procedure in Algorithm~\ref{alg:bfgs}(c) ensures that
$B^{n+1}$ remains symmetric and positive definite. The initial value $B^0$
is chosen as
\begin{align*}
  B_0 = \frac{-1}{\alpha} \Delta_h^{-1} \colon \vec{D}_h \rightarrow \vec{D}_h
\end{align*}
in order to exactly recover the inverse Hessian of the control cost as
suggested by the local convergence theory outlined
in~\cite{Griewank:1987,KelleySachs:1991}.


\section{Numerical illustrations}
\label{sec:numerical_examples}
In this section, we discuss a number of numerical results to illustrate the
performance of the shape optimization algorithm.
We implemented the algorithm in a small C++ code using the optimization
toolkit \texttt{DOpElib}~\cite{dopelib} which is based on the finite
element library \texttt{deal.II}~\cite{dealIIcanonical,dealII94}. We have
made our source code publicly available on
Zenodo\footnote{\url{https://zenodo.org/records/10459309}}
\cite{bezbaruah2024}.
For the sake of simplicity
we restrict our numerical computations to 2D by assuming translation
invariance in the $z$-direction. Then, the cell problem
\eqref{eq:cell_problem_ale} and the averaging \eqref{eq:eff_ale} reduce to
a 2D problem. Throughout the section we have chosen $\varepsilon=\unit$ and
the surface conductivity to be given by \eqref{eq:drude} with a fixed
frequency of either $\omega=0.3$, or $\omega=0.5$. The reference geometry
consists of an inscribed circle $\hat \Sigma$ at the center of $\hat
Y=[0,1]^2$ with a radius of $r=0.3$; see Figure~\ref{fig:deformation}. We
have chosen a fixed spatial discretization $\mathcal{T}_h$ of $53\,248$
quadrilaterals (fitted to the hypersurface $\hat\Sigma$) which amounts to
$429\,062$ degrees of freedom for the (complex tensor-valued) state problem
and $107\,266$ for the (vector-valued) control problem.

\subsection{Influence of the regularization parameters}
\label{subse:parameter_study}
\begin{table}[t]
    \centering
    \begin{tabular}{lllcccc}
      \toprule
      $\boldsymbol\beta$ & $\boldsymbol\alpha$ & $\boldsymbol\alpha_\Sigma$ & $\eff_{xx}$ & $\eff_{xy}$ &
      deviation & optimality\\[0.25em]
      0.2  &  0.01 & 10  & $0.8244+0.0200\im$ & $0.0444-0.0174\im$ & $4.63\%$ & $0.00503$ \\
      0.2  & 0.001 & 100 & $0.8233+0.0202\im$ & $0.0441-0.0170\im$ & $4.51\%$ & $0.00739$ \\
      0.2  & 0.001 & 10  & $0.8057+0.0133\im$ & $0.0484-0.0090\im$ & $1.70\%$ & $0.00890$ \\[0.2em]
      0.1  &  0.01 & 10  & $0.8239+0.0204\im$ & $0.0440-0.0174\im$ & $4.61\%$ & $0.02082$ \\
      0.1  & 0.001 & 100 & $0.8230+0.0206\im$ & $0.0438-0.0171\im$ & $4.51\%$ & $0.02598$ \\
      0.1  & 0.001 & 10  & $0.8052+0.0140\im$ & $0.0481-0.0090\im$ & $1.72\%$ & $0.01366$ \\[0.2em]
      0.05 &  0.01 & 10  & $0.8235+0.0207\im$ & $0.0438-0.0173\im$ & $4.59\%$ & $0.02360$ \\
      0.05 & 0.001 & 100 & $0.8228+0.0209\im$ & $0.0436-0.0170\im$ & $4.51\%$ & $0.07232$ \\
      0.05 & 0.001 & 10  & $0.8049+0.0145\im$ & $0.0480-0.0090\im$ & $1.73\%$ & $0.00472$ \\
      \bottomrule
    \end{tabular}
    \caption{Diagonal and off-diagonal components of $\eff(\hvchi,\hvq)$,
      relative deviation $\big\|\eff - \eps^{\text{target}}
      \big\|_{\text{Fr.}} /\big\|\eps^{\text{target}}\big\|_{\text{Fr.}}$,
      optimality, \ie, relative norm of the reduced gradient, for
      different values of stabilization parameters $\beta$, $\alpha$ and
      $\alpha_\Sigma$ (rows). Results are shown for a fixed number of 200
      iterations of the BFGS algorithm~\ref{alg:bfgs}. The initial deviation was
      $10.97\%$.}
    \label{tab:error_table}
    \begin{tabular}{lllcccc}
      \toprule
      $\boldsymbol\beta$ & $\boldsymbol\alpha$ & $\boldsymbol\alpha_\Sigma$ & $\eff_{xx}$ & $\eff_{xy}$ &
      deviation & optimality\\[0.25em]
      0.2  & 0.001 & 10  & $0.8057+0.0135\im$ & $0.0483-0.0092\im$ & $1.73\%$ & $0.00667$ \\
      0.1  & 0.001 & 10  & $0.8052+0.0141\im$ & $0.0481-0.0092\im$ & $1.74\%$ & $0.00747$ \\
      0.05 & 0.001 & 10  & $0.8048+0.0145\im$ & $0.0480-0.0091\im$ & $1.75\%$ & $0.00533$ \\
      \bottomrule
    \end{tabular}
    \caption{Subset of the parameter study reported in
      Table~\ref{tab:error_table} but with 5 instead of 6 global refinement
      steps resulting in $13\,312$ quadrilaterals which amounts to
      $108\,040$ degrees of freedom for the (complex tensor-valued) state
      problem and in $27\,010$ for the (vector-valued) control problem.}
    \label{tab:error_table_2}
\end{table}
We first present a parameter study to assess the influence of the
regularization parameters $\alpha$ and $\beta$ found in
\eqref{eq:cost_functional} on the target tensor $\eff$ and the deformation
field $\hat\vq$. The optimization problem without stabilization terms
in \eqref{eq:cost_functional} is highly ill-posed; the main reason being
the fact that the deformation vector $\hat\vq$ has no influence on the
target functional away from the interface $\hat\Sigma$. Thus, a reasonable
amount of penalization is required to (a) ensure consistent mesh regularity
(\ie \, $\hat J$ being reasonably close to~1), and (b) allow the geometry
to deform sufficiently to actually obtain an effective tensor
$\eff(\hvchi,\hvq)$ close to the target $\eps^{\text{target}}$. For the
parameter study we choose a target permittivity tensor of
\begin{align*}
  \eps^{\text{target}} =
  \begin{pmatrix}
    0.8 + 0.008\im &  0.05 \\
    0.05 & 0.8 + 0.008\im \\
  \end{pmatrix},
\end{align*}
which has a moderate initial relative deviation $\big\|\eff -
\eps^{\text{target}} \big\|_{\text{Fr.}} /
\big\|\eps^{\text{target}}\big\|_{\text{Fr.}}$ of around $10.97\%$. We
choose to perform a fixed number of 200 steps of Algorithm~\ref{alg:bfgs}
without an active stopping criterion. Results for varying degrees of
regularization $\alpha=0.01$, $0.001$, $\alpha_\Sigma=10$, $100$, and
$\beta=0.2$, $0.1$, $0.05$ are reported in Table~\ref{tab:error_table}.
We see in the above experiments that the combination of $\alpha$ and
$\alpha_\Sigma$ have the largest influence on the achieved deviation from
the target tensor. Lower values of the stabilization parameters result in
smaller deviations; see Table \ref{tab:error_table}. As a last figure of
merit we also report the achieved \emph{optimality}, \ie, the norm of the
reduced gradient normalized over the initial value: $\|\vec{\delta
c}_h(\hvq_h^n)\|/\|\vec{\delta c}_h(\hvq_h^0)\|$ for the final step
$n=200$. Here, we observe that the highest reduction after 200 steps with
an optimality of around $0.005$ is achieved for the choice $\alpha=0.001$,
$\alpha_\Sigma=10$, $\beta=0.05$. However, if the stabilization parameters
are chosen too small, the mesh can degrade, in particular near the edges of
the interface $\Sigma_h$. Thus, in order to balance both these factors, we
make a conservative choice of $\alpha=0.001$, $\alpha_\Sigma=10$ and
$\beta=0.1$ for all subsequent numerical tests.

As a final test we examine the influence of mesh refinement on the
numerical result and rerun the case of $\alpha=0.001$, $\alpha_\Sigma=10$
and $\beta=0.2$, $0.1$, $0.05$ with a lower resolution of $13\,312$
quadrilaterals resulting in $108\,040$ degrees of freedom for the (complex
tensor-valued) state problem and in $27\,010$ for the (vector-valued)
control problem; see Table~\ref{tab:error_table_2}. The final $\eff$ values
after 200 iterations are very close to the results obtained for 6 global
refinement steps; cf Table~\ref{tab:error_table}. We conclude that the
chosen resolution of $53\,248$ quadrilaterals is appropriate with minimal
influence on the optimization result.

\subsection{Optimizing an epsilon-near-zero material}
\label{subse:enz_example}
We now illustrate the shape optimization procedure for three different
target permittivity tensors given by
\begin{align*}
  \eps^{\text{target}} =
  \begin{pmatrix}
    \ast &  0.05 \\
    0.05 & 0.5 + 0.01\im \\
  \end{pmatrix},
\end{align*}
where we vary the $\eps^{\text{target}}_{xx}$ component from (a)
$0.5+0.01\im$, (b) $0.25+0.005\im$, to (c) $0.0$. The target tensor has been
chosen close to the initial permittivity tensor of
\begin{align*}
  \eff_{\text{ref}} =
  \begin{pmatrix}
     0.50304+0.01114\im & 0.0 \\
     0.0                & 0.50304+0.01114\im \\
  \end{pmatrix},
\end{align*}
obtained for the reference configuration with frequency $\omega=0.3$. As
the target vector gradually gets closer to an epsilon-near-zero material
\cite{mattheakis2016,maier19c} an increasingly larger mesh deformation is
required to achieve an optimal configuration. We chose to add an
off-diagonal value of $0.05$ in the target permittivity tensor to force an
increased interaction between the $x$- and $y$-directions. For our
numerical computation we use the stabilization parameters discussed in
Section~\ref{subse:parameter_study} and a stopping criterion to achieve a
reduction of the reduced gradient, viz. $\|\vec{\delta
c}_h(\hvq_h^n)\|/\|\vec{\delta c}_h(\hvq_h^0)\|$, of better than $10^{-4}$.
\begin{figure}[tb]
  \centering
  \subfloat[$\eps^{\text{trgt.}}_{xx}=0.5+0.01\im$]{
    \includegraphics[width = 0.232\textwidth]{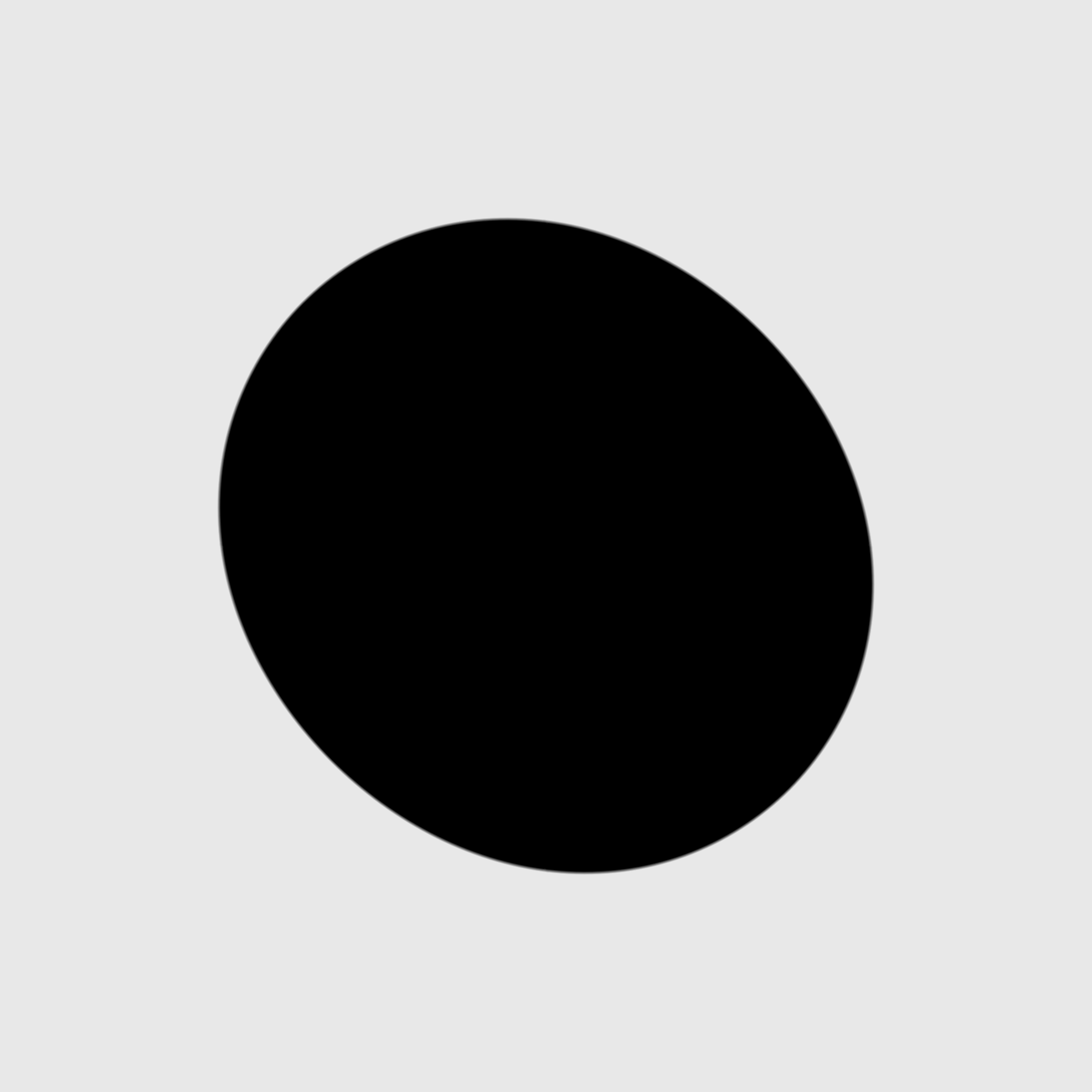}
  }
  \subfloat[$\eps^{\text{trgt.}}_{xx}=0.25+0.00\im$]{
    \includegraphics[width = 0.232\textwidth]{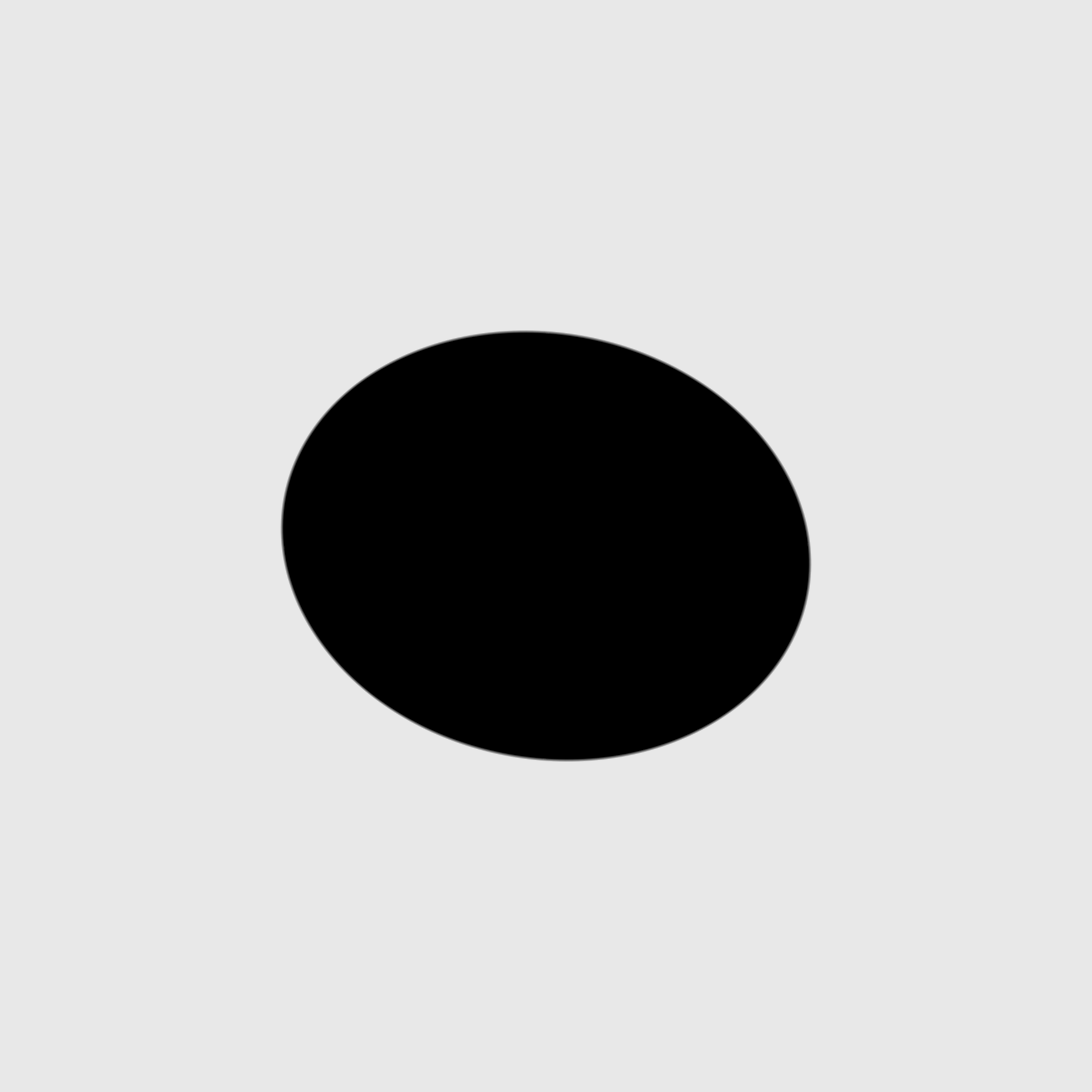}
  }
  \subfloat[$\eps^{\text{trgt.}}_{xx}=0.0$]{
    \includegraphics[width = 0.232\textwidth]{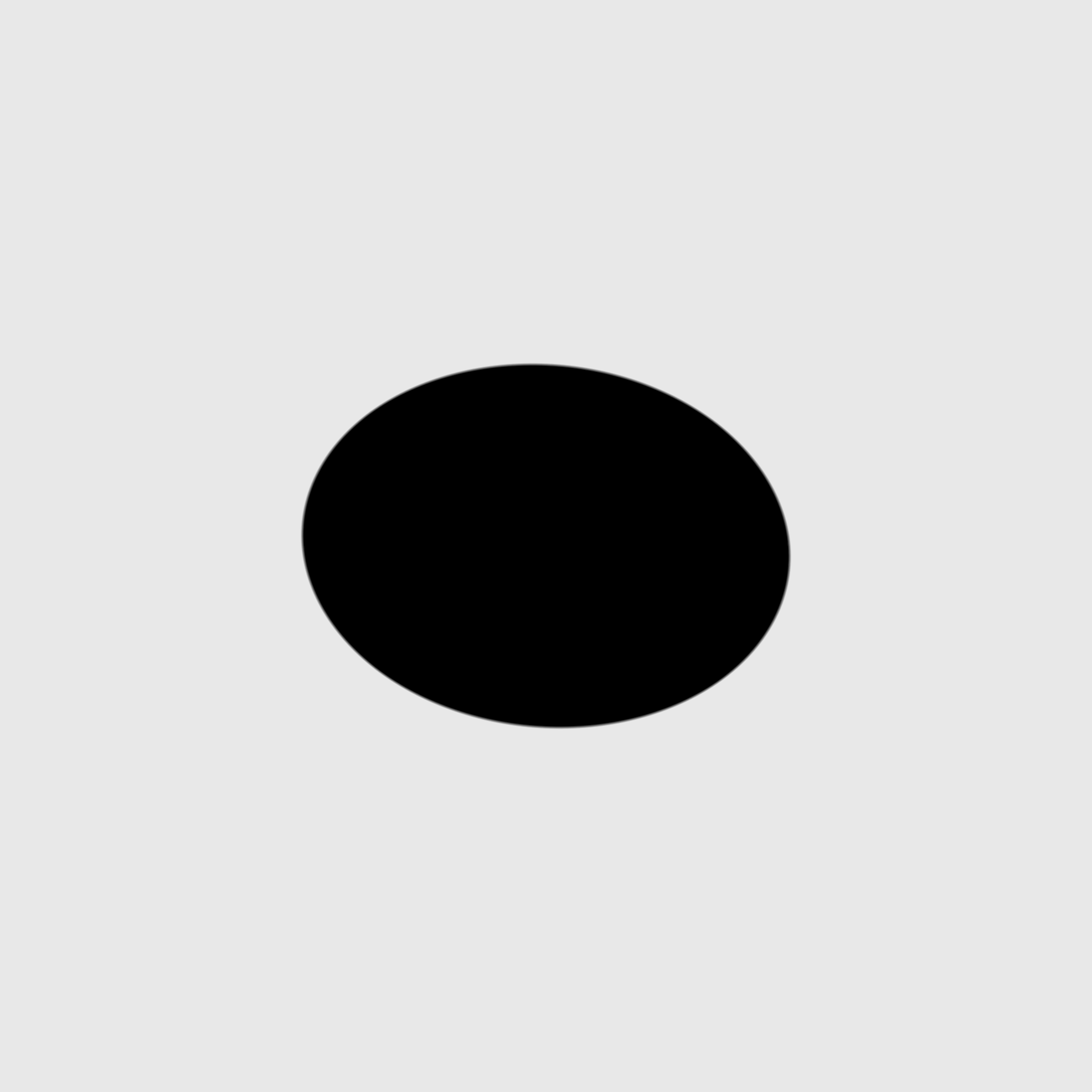}
  }
  \subfloat[reference]{
    \includegraphics[width = 0.232\textwidth]{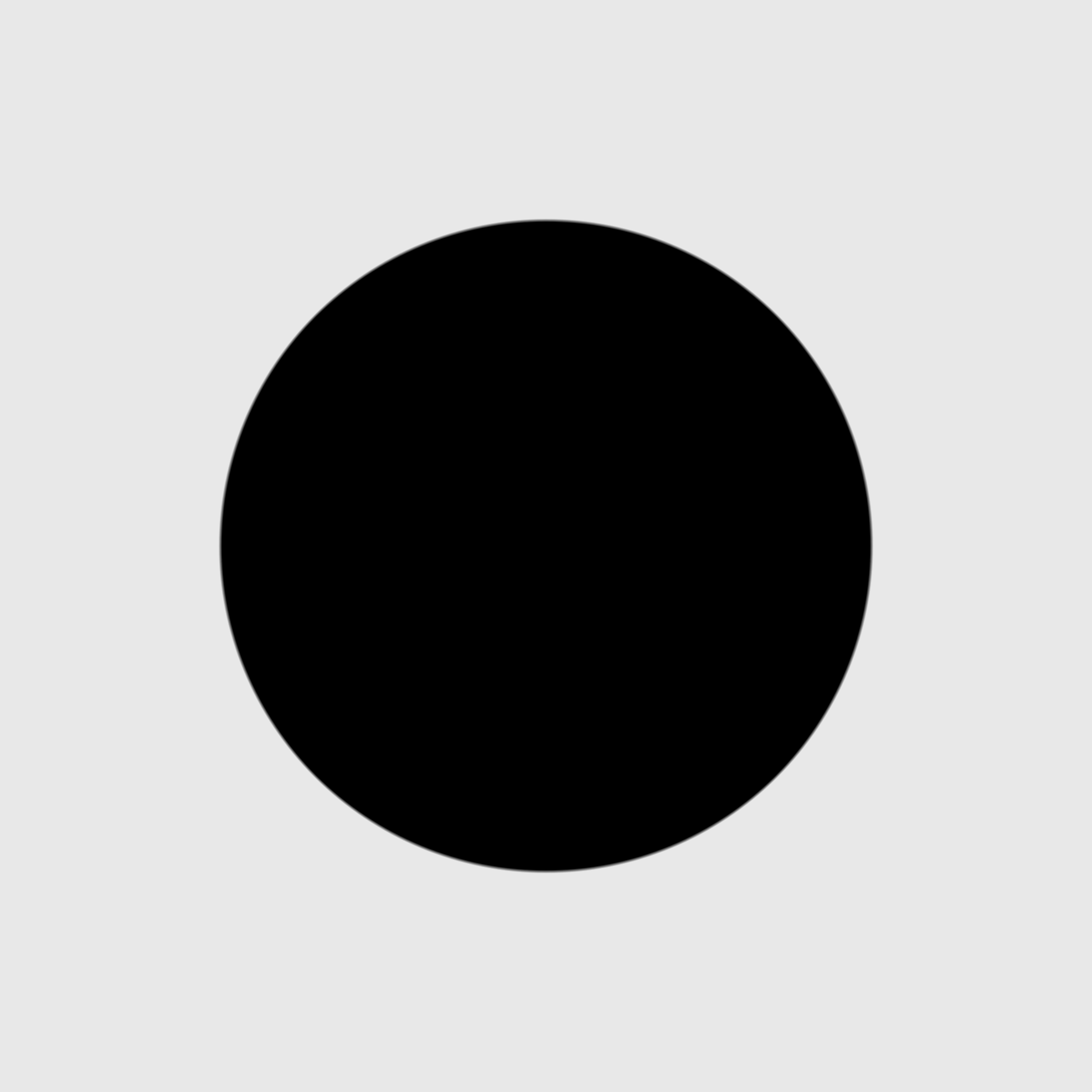}
  }

  \subfloat[$\eps^{\text{trgt.}}_{xx}=0.5+0.01\im$]{
    \includegraphics[width = 0.232\textwidth]{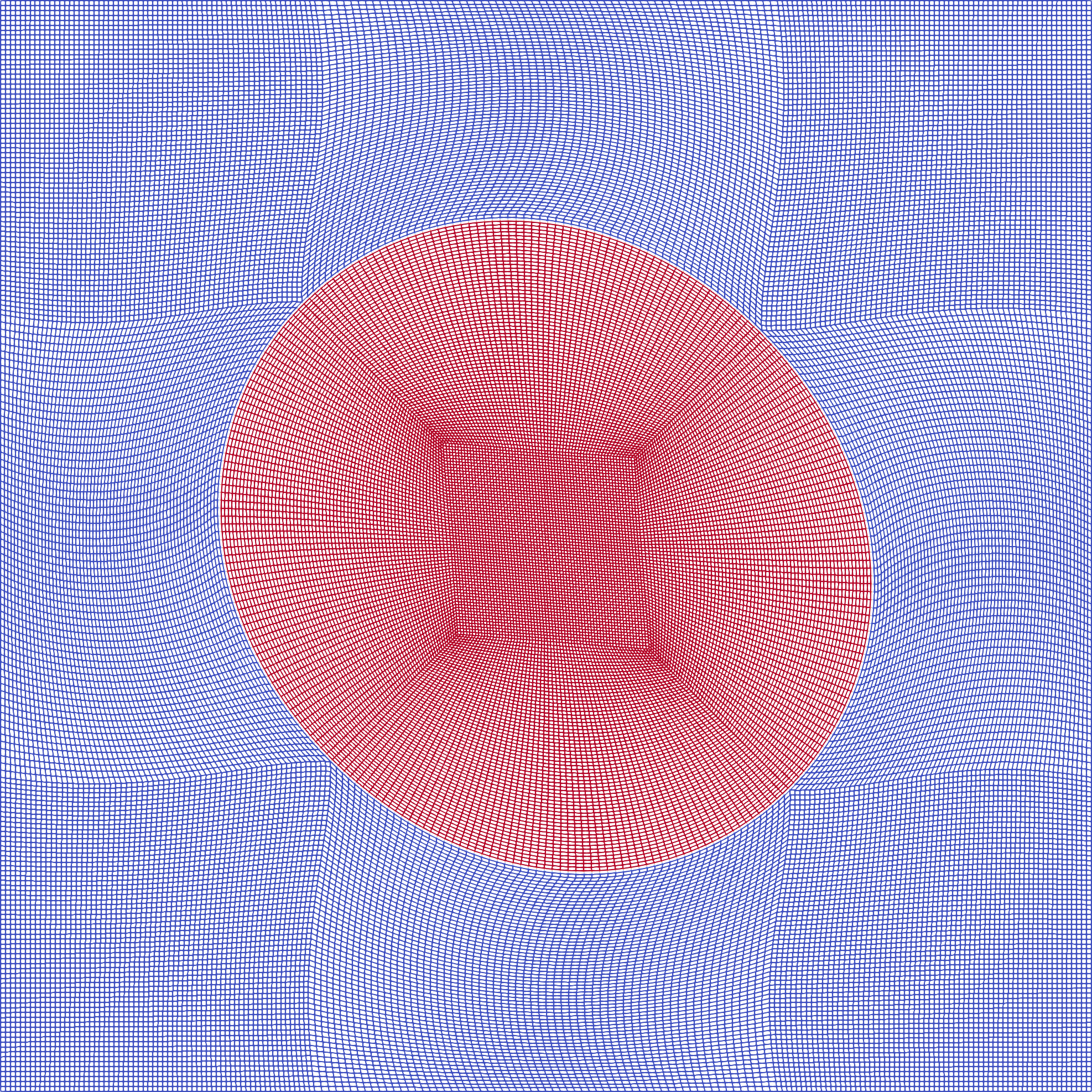}
  }
  \subfloat[$\eps^{\text{trgt.}}_{xx}=0.25+0.00\im$]{
    \includegraphics[width = 0.232\textwidth]{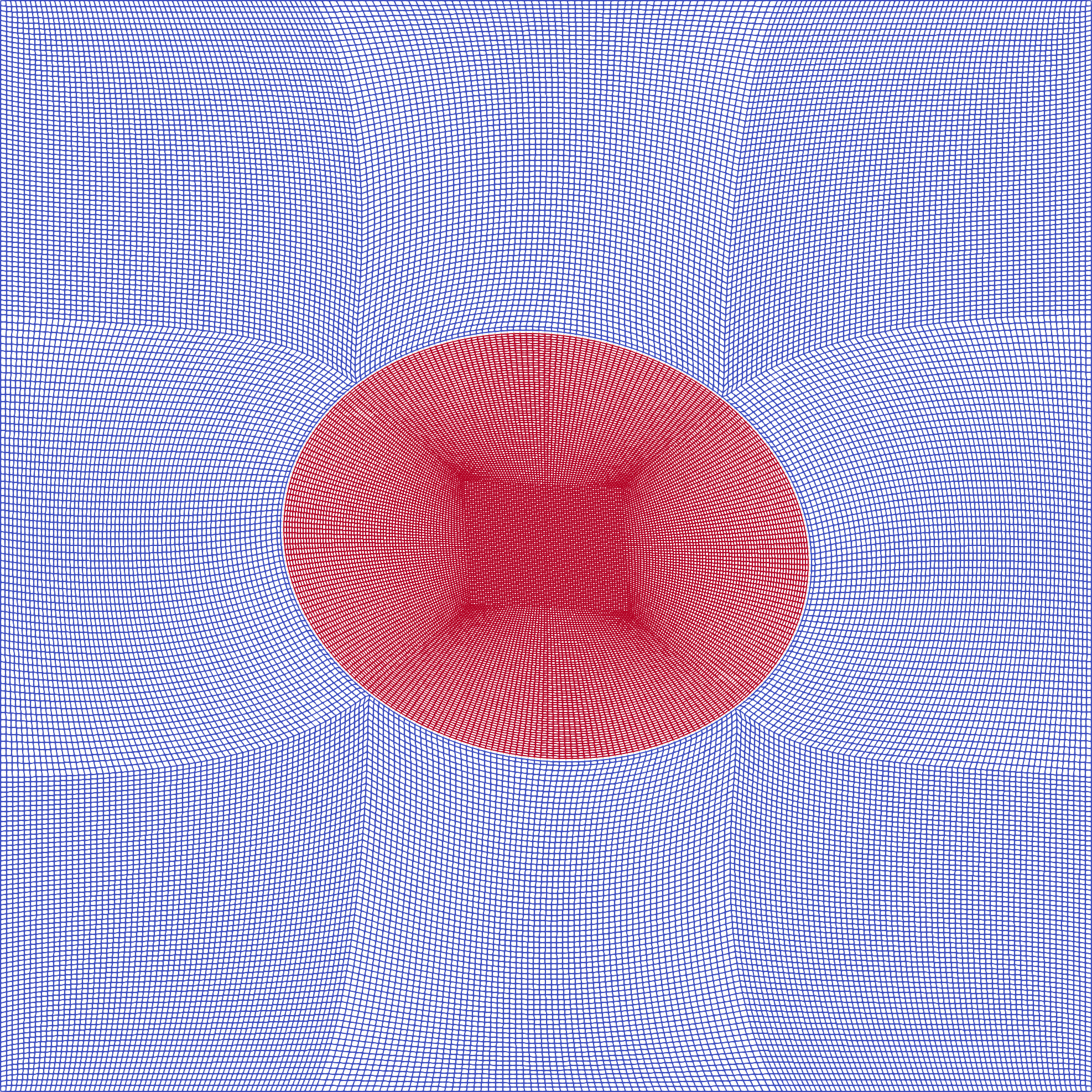}
  }
  \subfloat[$\eps^{\text{trgt.}}_{xx}=0.0$]{
    \includegraphics[width = 0.232\textwidth]{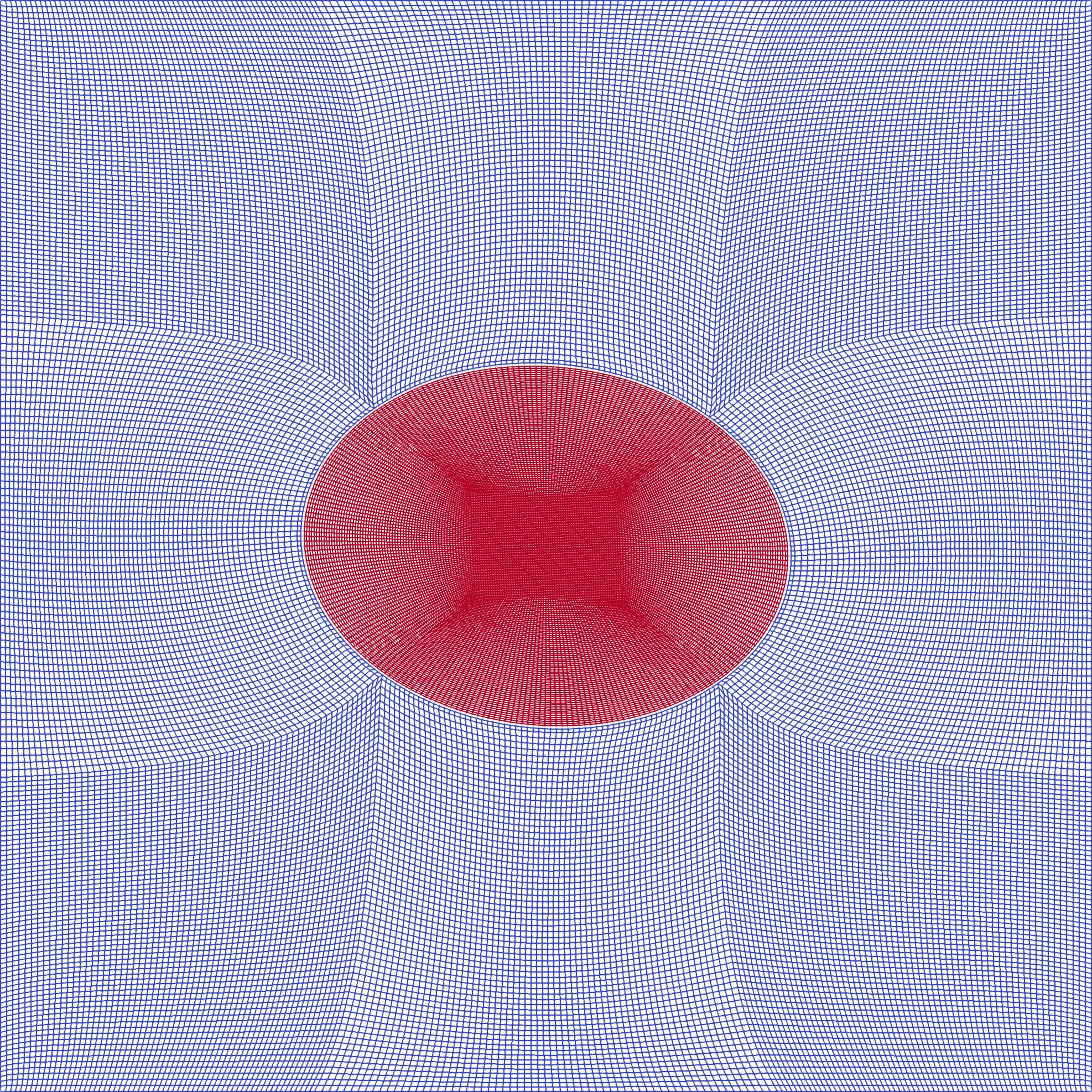}
  }
  \subfloat[reference]{
    \includegraphics[width = 0.232\textwidth]{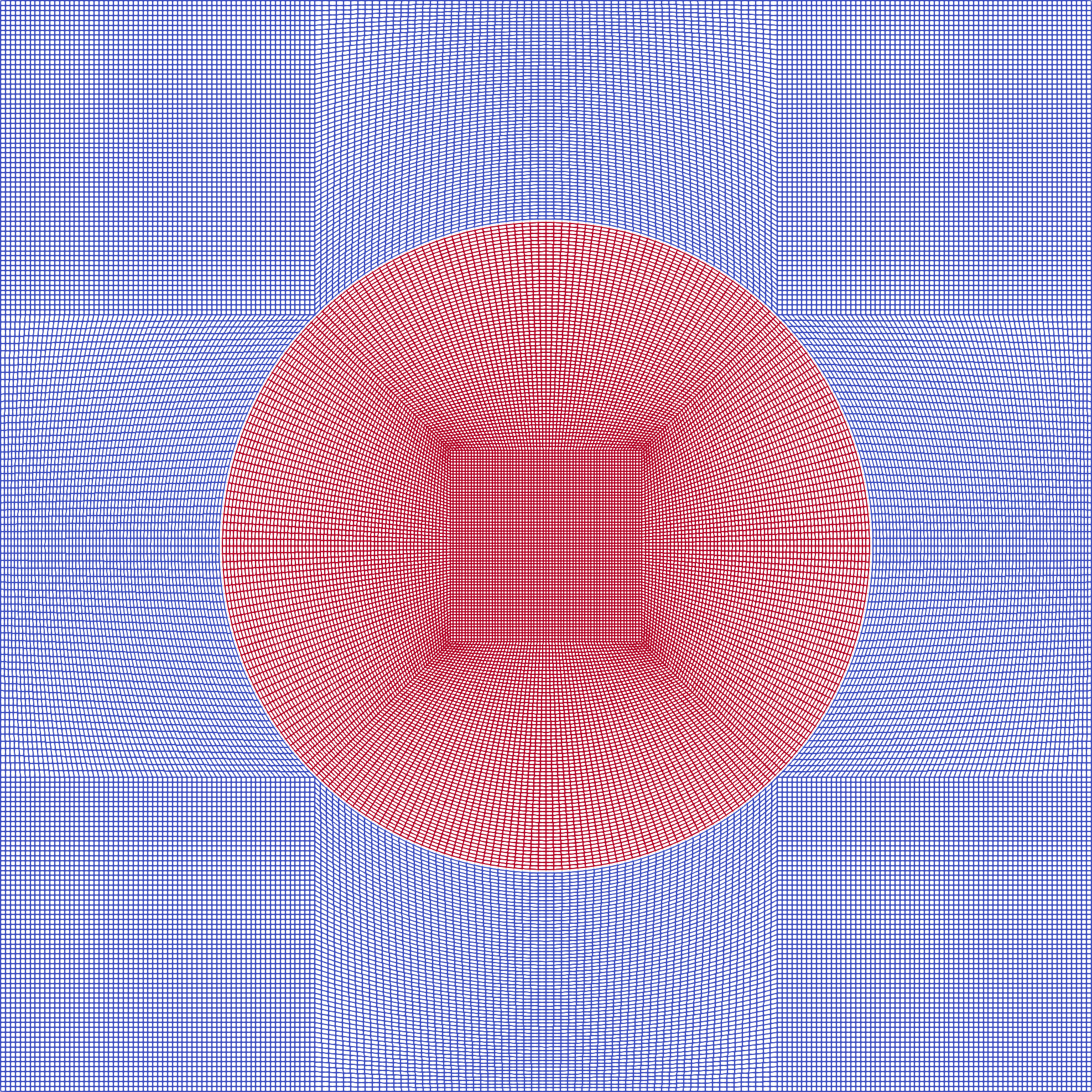}
  }
  \caption{Epsilon-near-zero testcase: Final geometry obtained for target
    cases (a), (b) and (c) with increasingly smaller
    $\eps^{\text{target}}_{xx}$ component. The corresponding deformed (and
    initial) meshes are shown in (e)-(h). The black region in (a)-(d), as
    well as the red region in (e)-(h) show the volume surrounded by the
    interface $\Sigma_h$.}
  \label{fig:epsilon_near_zero}
\end{figure}
\begin{table}[tb]
    \begin{tabular}{rcccccl}
      \toprule
           & $\eff_{xx}$        & $\eff_{xy}$        & $\eff_{yy}$        &
           initial  & final    & steps  \\[0.25em]
      ref. & $0.5030+0.011\im$ & $0.0-0.0\im$      & $0.5030+0.011\im$ \\
      (a)  & $0.5041+0.011\im$ & $0.0488-0.001\im$ & $0.5041+0.011\im$ & $7.09\%$ & $0.65\%$ & $263$  \\
      (b)  & $0.2897+0.027\im$ & $0.0486-0.002\im$ & $0.5285+0.016\im$ & $26.3\%$ & $5.40\%$ & $2047$ \\
      (c)  & $0.0271+0.054\im$ & $0.0493-0.003\im$ & $0.5161+0.021\im$ & $50.8\%$ & $6.36\%$ & $2977$ \\
      \bottomrule
    \end{tabular}
    \caption{Epsilon-near-zero testcase: Final permittivity tensors
      obtained for target cases (a), (b), (c), and the starting value for
      the undeformed reference configuration. In addition, we report the
      initial and final deviation $\big\|\eff - \eps^{\text{target}}
      \big\|_{\text{Fr.}} / \big\|\eps^{\text{target}}\big\|_{\text{Fr.}}$,
      as well as the number of BFGS iterations in algorithm~\ref{alg:bfgs}
      needed to achieve convergence.}
    \label{tab:epsilon_near_zero}
\end{table}
\begin{figure}[tb]
  \centering
  \input{optimality-bfgs.tex}
  \caption{
    Evolution of the optimality, $\|\vec{\delta
    c}_h(\hvq_h^n)\|/\|\vec{\delta c}_h(\hvq_h^0)\|$, during the BFGS
    solution process for the three cases (a) $0.5+0.01\im$, (b)
    $0.25+0.005\im$, to (c) $0.0$. The thick line for each case is a
    smoothed Bezier curve (gnuplot builtin) that is overlayed over the
    actual, oscillatory value.
  }
  \label{fig:epsilon_near_zero_optimality}

  \input{deviation-bfgs.tex}
  \caption{
    Evolution of the deviation, $\big\|\eff - \eps^{\text{target}}
    \big\|_{\text{Fr.}} / \big\|\eps^{\text{target}}\big\|_{\text{Fr.}}$,
    during the BFGS solution process for the three cases (a) $0.5+0.01\im$,
    (b) $0.25+0.005\im$, to (c) $0.0$.
  }
  \label{fig:epsilon_near_zero_deviation}
\end{figure}
The initial and final geometry is illustrated in
Figure~\ref{fig:epsilon_near_zero}. It can be seen that proceeding from
case (a) to (c) an increasingly larger mesh deformation is required, the
shapes remain largely elliptic. The final permittivity tensor values and
deviation are given in Table~\ref{tab:epsilon_near_zero}. With the chosen
stabilization parameters we were able to improve the initial deviation to
our target tensor consistently by an order of magnitude.
\begin{remark}
  We point out that it does not seem possible in general to achieve a
  deviation of zero for arbitrary target permittivities \cite{maier19c}.
  This is largely due to the fact that the effective permittivity tensor
  $\eff$ as a function of shape and frequency possesses a well defined
  structure that does not allow to tune all components and the real and
  imaginary part arbitrarily; see \cite{maier19c}.
\end{remark}
In Figure~\ref{fig:epsilon_near_zero_optimality} we report the evolution of
the optimality, $\|\vec{\delta c}_h(\hvq_h^n)\|/\|\vec{\delta
c}_h(\hvq_h^0)\|$, during the solution process. We note that with the
larger deformation necessary for cases (b) and (c) a significant increase
in the number of required steps to achieve convergence can be observed. As
a final figure of merit we examine the evolution of the deviation,
$\big\|\eff - \eps^{\text{target}} \big\|_{\text{Fr.}} /
\big\|\eps^{\text{target}}\big\|_{\text{Fr.}}$, during the solution
process; see Figure~\ref{fig:epsilon_near_zero_deviation}. We observe that
the deviation converges to its final value very fast (in between 10 to 100
steps) which corresponds to the initial steep slope for the optimality; cf.
Figure~\ref{fig:epsilon_near_zero_optimality}. The remaining steps are then
spend on further minimizing the penalty and thus improving the overall mesh
quality.

\subsection{Large mesh deformations}
\label{subse:large_deformation_example}
\begin{figure}[tb]
  \centering
  \subfloat[$\eps^{\text{trgt.}}_{xy}=0.10$]{
    \includegraphics[width = 0.232\textwidth]{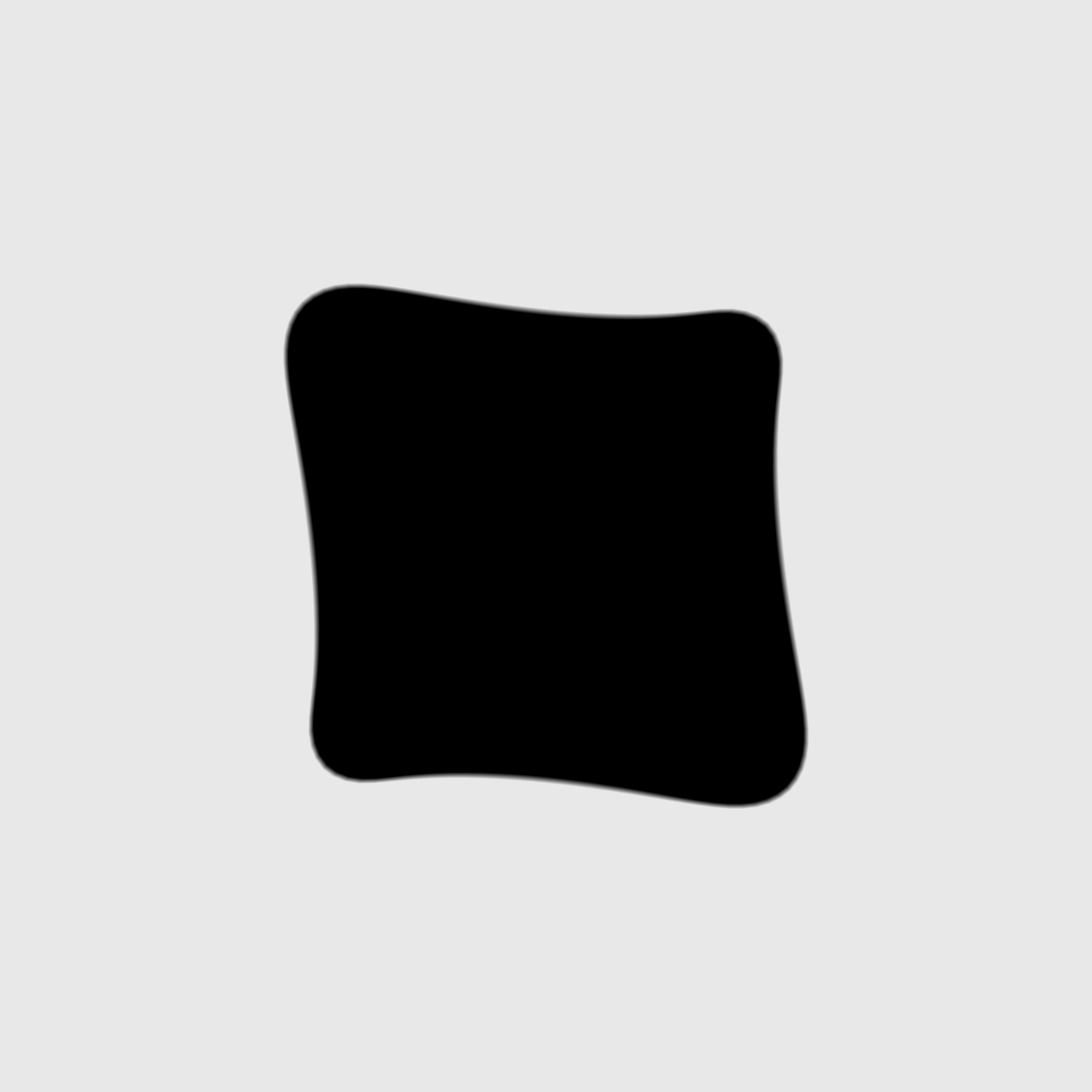}
  }
  \subfloat[$\eps^{\text{trgt.}}_{xy}=0.15$]{
    \includegraphics[width = 0.232\textwidth]{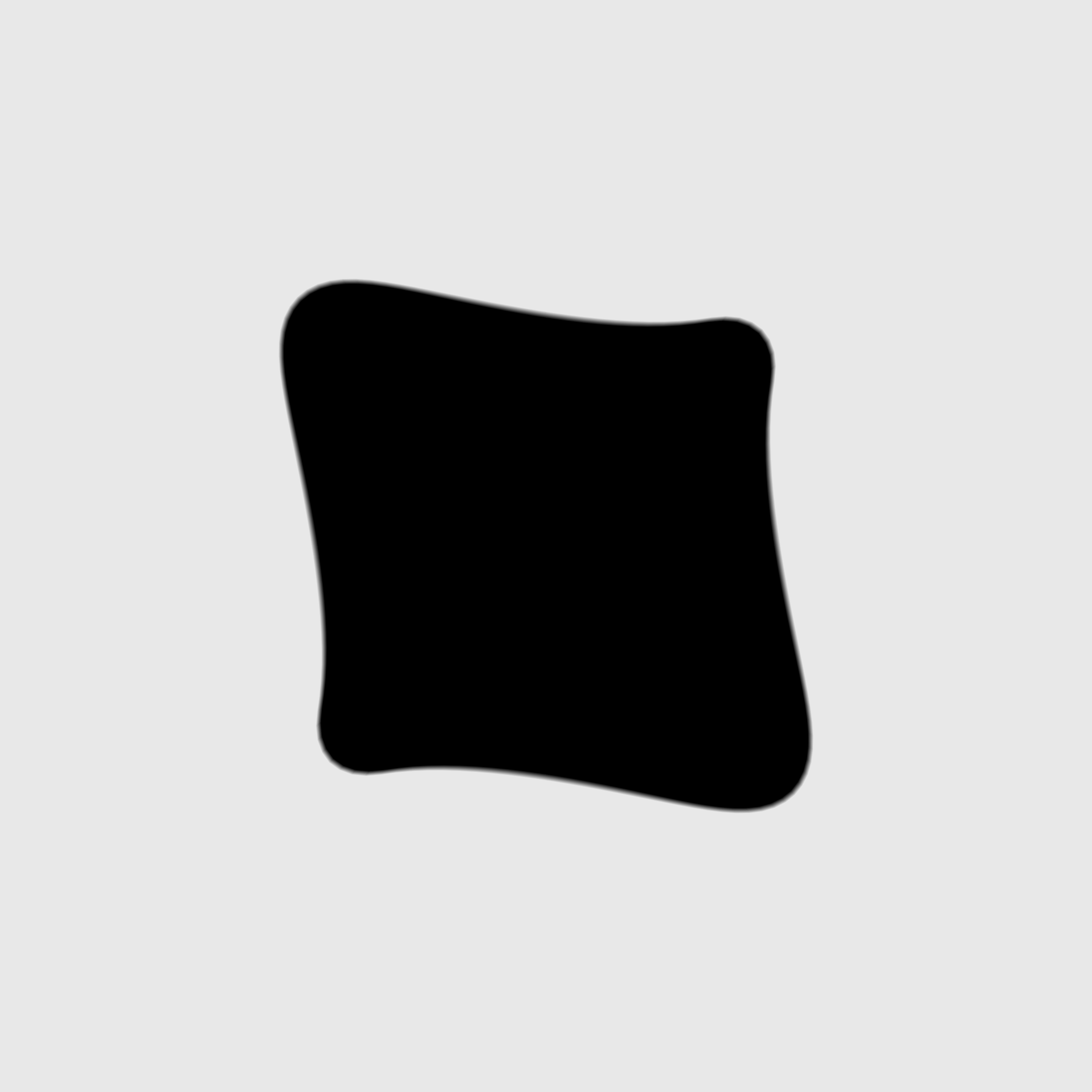}
  }
  \subfloat[$\eps^{\text{trgt.}}_{xy}=0.20$]{
    \includegraphics[width = 0.232\textwidth]{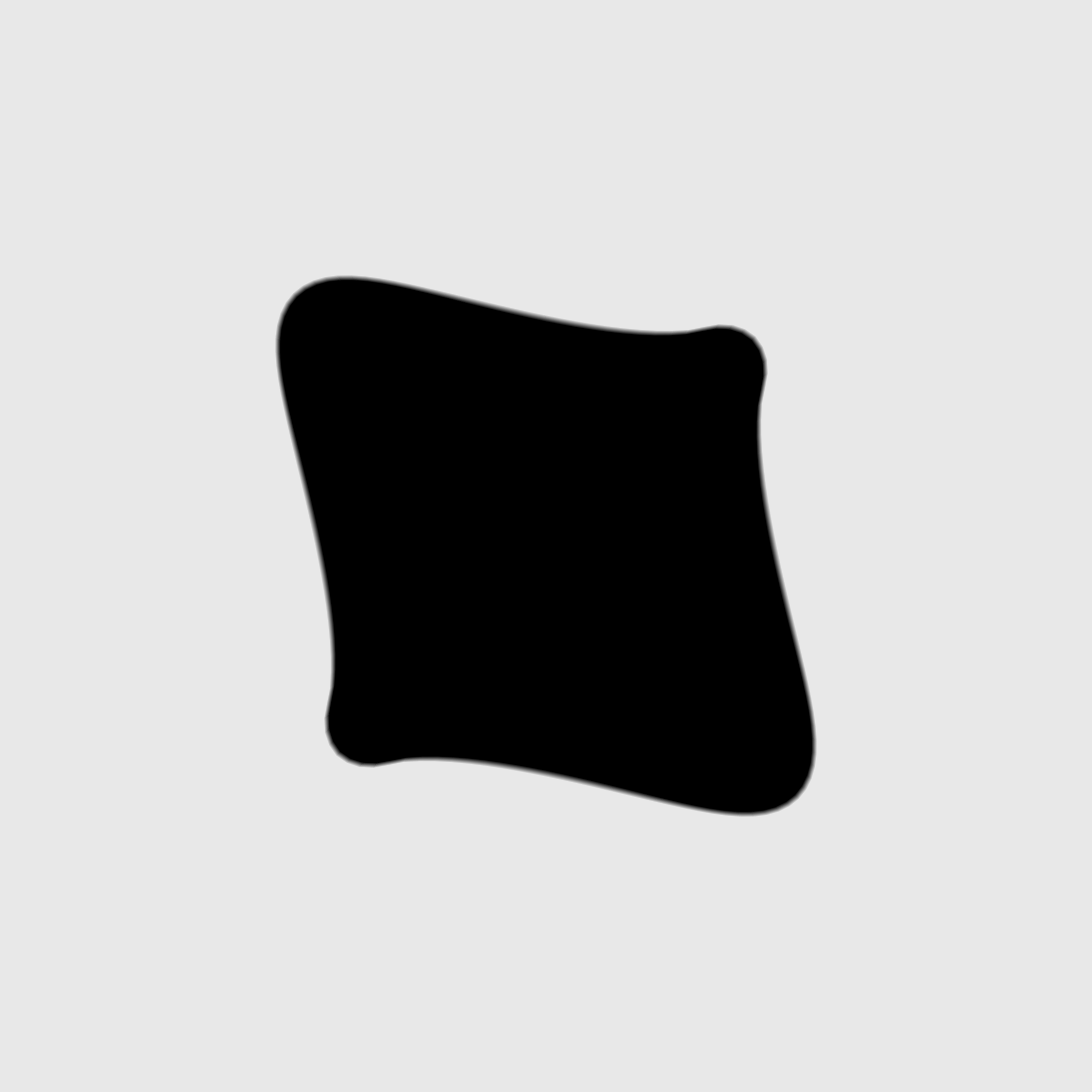}
  }
  \subfloat[reference]{
    \includegraphics[width = 0.232\textwidth]{target-start.png}
  }
  \caption{%
    Large deformation testcase: Final geometry obtained for target cases
    (a), (b) and (c) with increasingly larger $\{\eps^{\text{target}}_{xy},
    \eps^{\text{target}}_{yx}\}$ components.}
  \label{fig:large_deformation}
\end{figure}
\begin{table}[tb]
  \begin{center}
    \begin{tabular}{rccccr}
      \toprule
           & $\eff_{xx}$        & $\eff_{xy}$      & initial  & final    &
           steps  \\[0.2em]
      ref. & $0.7783+0.003\im$ & $0.0-0.0\im$      & \\
      (a)  & $0.5254+0.063\im$ & $0.0953-0.013\im$ & $41.8\%$ & $8.58\%$ & $0+512$ \\
      (b)  & $0.5256+0.063\im$ & $0.1429-0.020\im$ & $44.7\%$ & $8.89\%$ & $100+435$ \\
      (c)  & $0.5258+0.063\im$ & $0.1904-0.027\im$ & $48.4\%$ & $9.30\%$ & $100+709$ \\
      \bottomrule
    \end{tabular}
  \end{center}
  \caption{%
    Large deformation testcase: Final permittivity tensors obtained for
    target cases (a), (b), (c), and the starting value for the undeformed
    reference configuration. In addition, we report the initial and final
    deviation, as well as the number of BFGS iterations in
    Algorithm~\ref{alg:bfgs} needed to achieve convergence.}
  \label{tab:large_deformation}
\end{table}
As a final test, we demonstrate that our optimization algorithm can handle
larger mesh deformations. For this, we use again the target tensor
introduced in Section~\ref{subse:enz_example} but vary the off-diagonal
elements instead of the $xx$-component:
\begin{align*}
  \eps^{\text{target}} =
  \begin{pmatrix}
    0.5 + 0.01 & \ast    \\
    \ast & 0.5 + 0.01\im \\
  \end{pmatrix},
\end{align*}
where we vary the $\eps^{\text{target}}_{xy}$, $\eps^{\text{target}}_{yx}$
components from (a) $0.10$, (b) $0.15$, to (c) $0.20$. In addition, we set
the angular frequency to $\omega=0.4$. For visualization purposes, we also
chose to run this set of computations with a lower resolution of $13\,312$
quadrilaterals (see also the discussion about mesh resolution in
Section~\ref{subse:parameter_study}).
The mesh deformation for cases (b) and (c) is so large that it would
require to increase the values of the stabilization parameters
significantly. This would result in a significantly higher deviation of the
obtained permittivity tensor from the target tensor possibly beyond what
would be deemed acceptable. We thus employ a slightly more sophisticated
strategy: In cases (b) and (c) we first run 100 steps of the BFGS algorithm
with a large penalty of $\beta=0.4$ for (b) and $\beta=0.8$ for (c) which
ensures that the mesh does not degrade too much and the regular solutions
shown in Figure~\ref{fig:large_deformation} are obtained. Afterwards we
fall back to the original stabilization value of $\beta=0.1$ and run the
BFGS algorithm until the stopping criterion is reached.

\begin{figure}[tb]
  \centering
  \subfloat[step 10, dev. $23\%$]{
    \includegraphics[width = 0.232\textwidth]{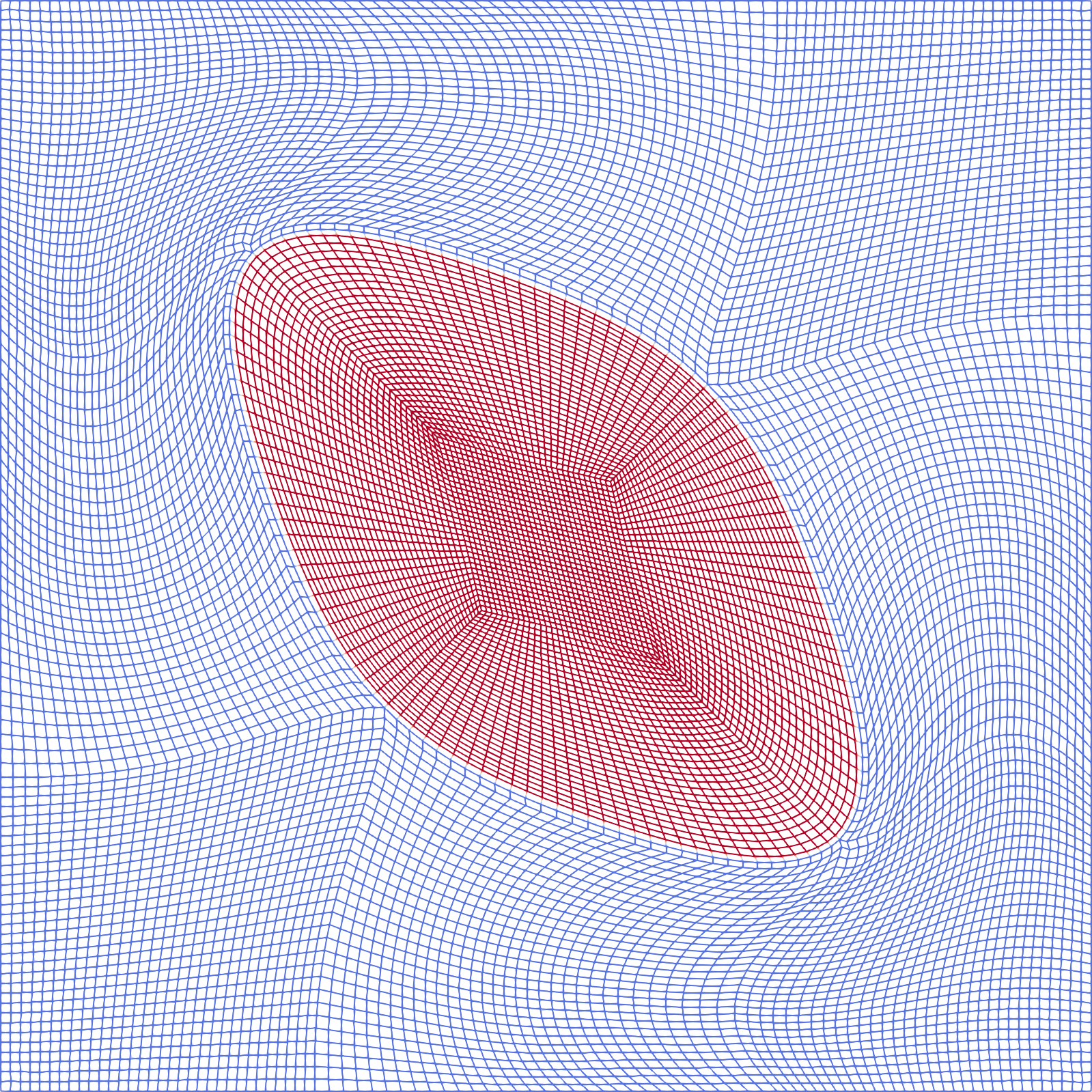}
  }
  \subfloat[step 30, dev. $21$]{
    \includegraphics[width = 0.232\textwidth]{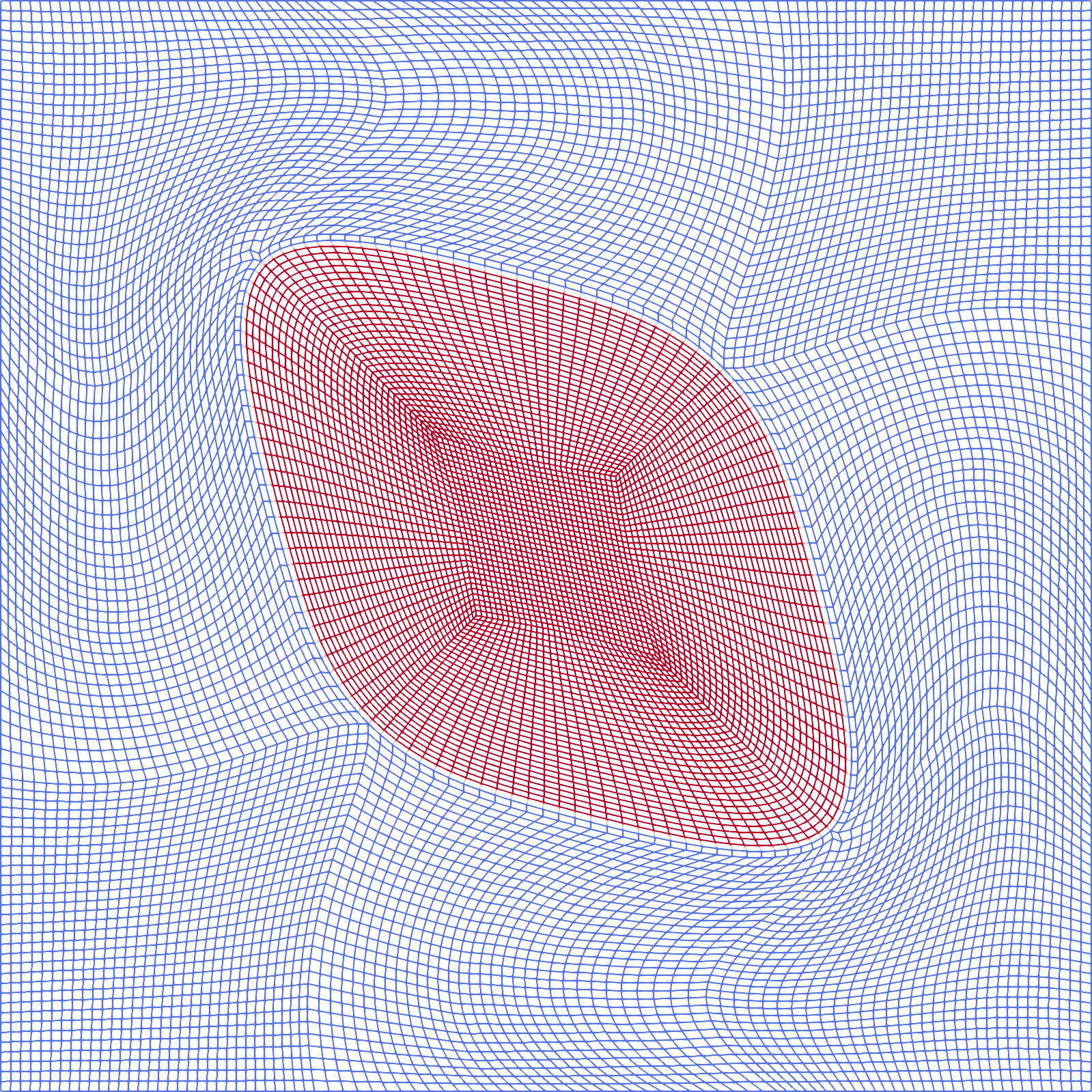}
  }
  \subfloat[step 100, dev. $16\%$]{
    \includegraphics[width = 0.232\textwidth]{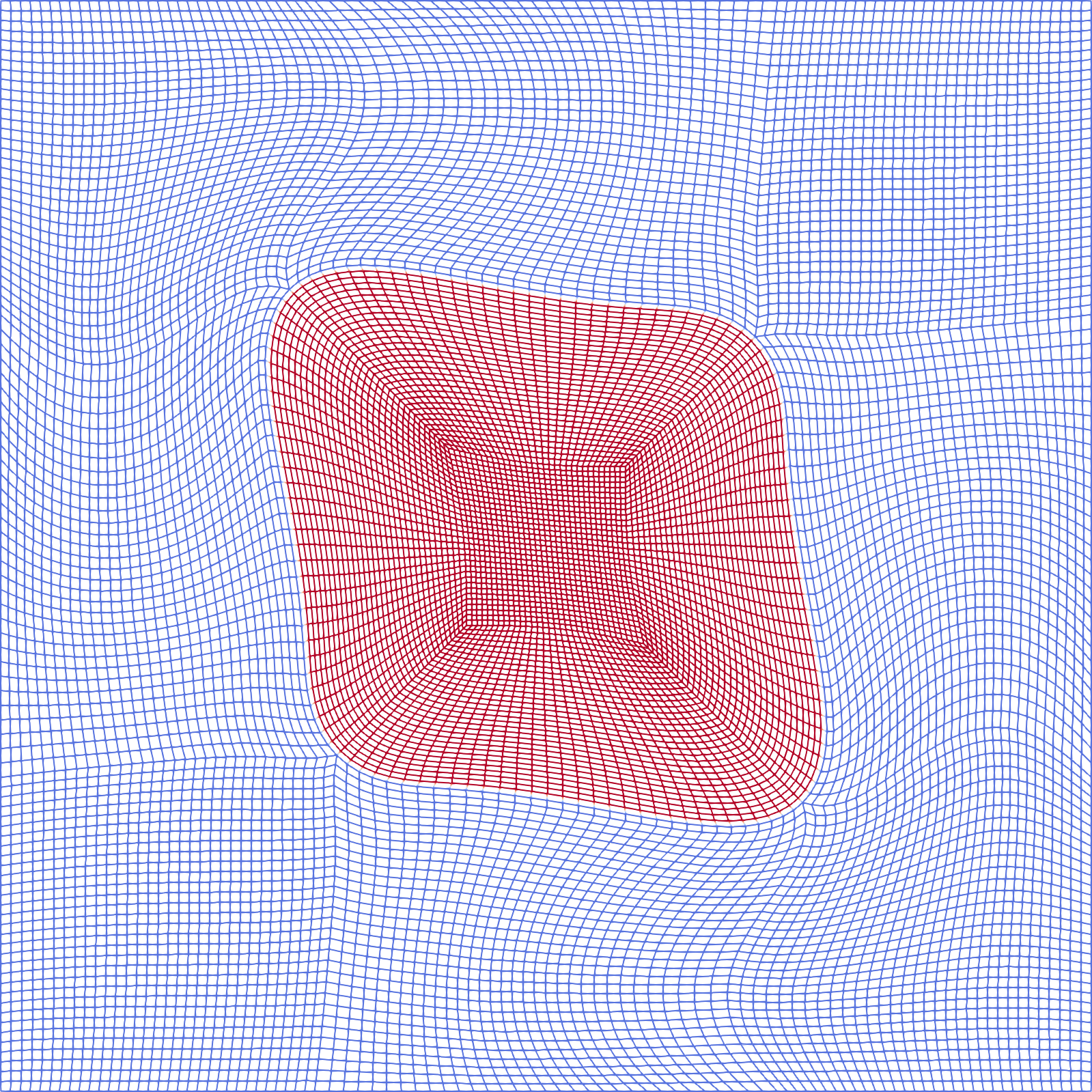}
  }
  \subfloat[st. 100+50, dev. $9\%$]{
    \includegraphics[width = 0.232\textwidth]{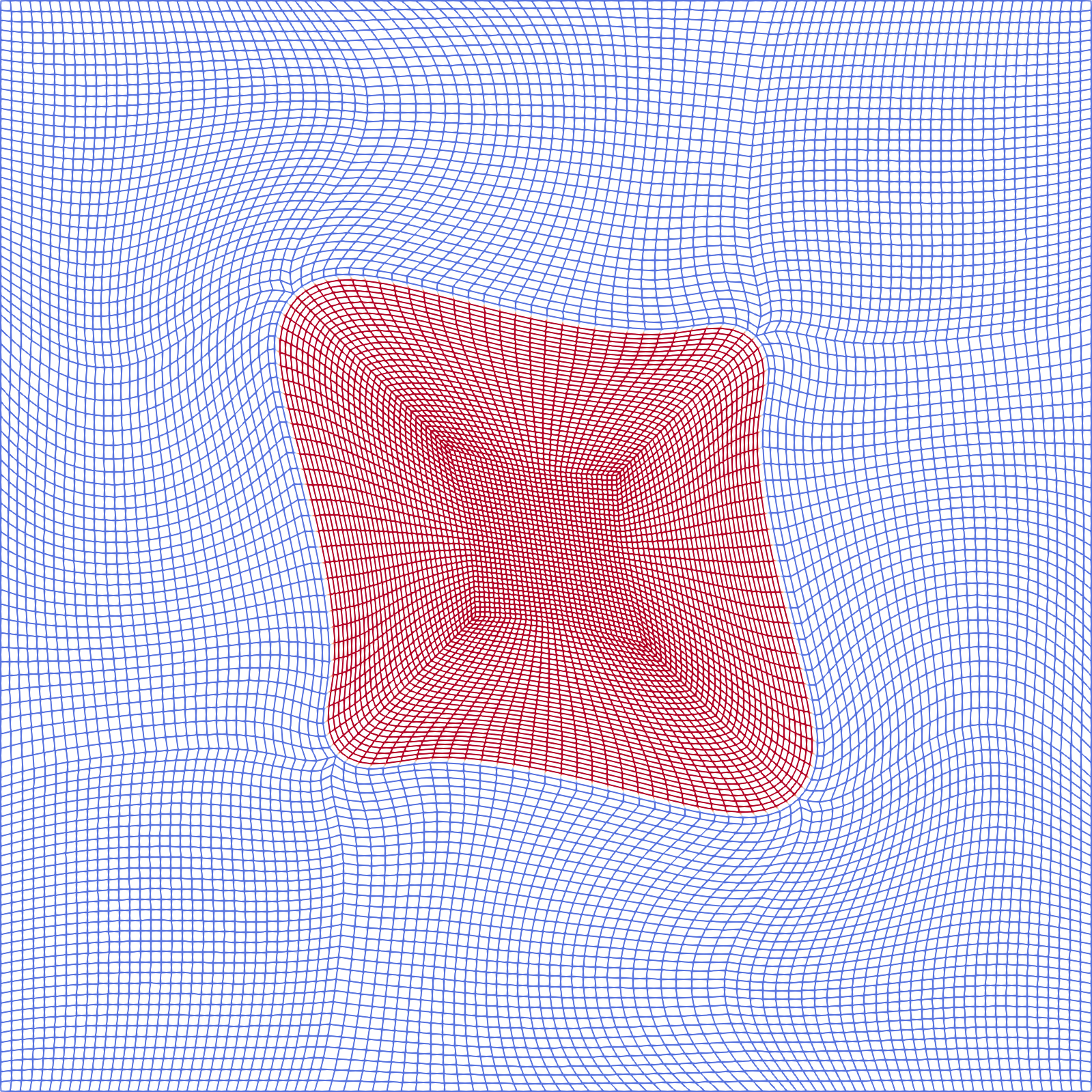}
  }
  \caption{Large deformation testcase: Mesh evolution for Case (c)
    with $\eps^{\text{target}}_{xy}=0.2$ after (a) 10, (b) 30, and (c) 100
    steps with a large regularization parameter $\beta=0.8$, and (d) after
    running another 50 steps with a smaller penalty of $\beta=0.1$. The
    deviation value reported is the normalized difference of the effective
    permittivity tensor to target tensor.}
  \label{fig:mesh_evolution}
\end{figure}
As reported in Figure~\ref{fig:large_deformation} with the modified choice
of target tensor and angular frequency a much more dramatic shape
deformation away from a simple ellipse was achieved. The corresponding
final permittivity tensors, final deviation and number of BFGS steps are
reported in Table~\ref{tab:large_deformation}. We reach a final deviation
of around $9\%$ for all three cases. We report in
Figure~\ref{fig:mesh_evolution} the mesh evolution observed during the
solution process for case (c) with $\eps^{\text{target}}_{xy}=0.2$. The
most crucial part in the solution process is between steps 30 and 100: if
we lower $\beta$ too much below $0.8$ then the algorithm tries to
approximate an eight-figure shape (and becoming singular in the process)
instead of reaching the shape shown in
Figure~\ref{fig:large_deformation}(c).


\section{Conclusion and outlook}
\label{sec:conclusion}
In this paper, a shape optimization problem for plasmonic crystals
consisting of dielectric inclusions was discussed. The objective was to
modify the shape of microscale inclusions so that the effective
permittivity tensor attained a set target. To achieve this goal, a mesh
deformation technique was introduced in the cell problem and the underlying
homogenization process. Furthermore, well-posedness and regularity of the
deformed cell problem were established.

Next, an optimization algorithm based on the
Broyden-Fletcher-Goldfarb-Shanno quasi-Newton method was presented. Through
a series of numerical experiments, the effectiveness of the algorithm was
demonstrated. Importantly, it was shown that the formulation and choice of
penalization effectively handled large mesh deformations, as evidenced in
the experimental results.

\emph{Outlook.}
The current approach is limited to optimizing a microstructure for a single
fixed frequency $\omega$. This seems somewhat limiting as in practice a
target permittivity that is valid over a (large) frequency interval is
desired. In this regard we will explore how the optimization approach can
be adapted to a result reported in \cite{maier19c}: The frequency response
of $\eff_{ij}(\omega)$ is described by a purely algebraic expression
\begin{align*}
  \eff_{ij}(\omega)\;=\;
  \varepsilon\,\delta_{ij}\;-\;
  \eta(\omega)\,N_{ij}
  \;-\;
  \sum_{n=1}^\infty\,
  \frac{\lambda_n\,\eta^2(\omega)}{\varepsilon-\lambda_n\,\eta(\omega)}
  \; M_{jn}
  \; M_{in},
\end{align*}
where $\eta(\omega)=\frac{\sigma(\omega)}{\im\omega}$, the coefficient
$N_{ij}$ is a weight only depending on the geometry, and $M_{in}$, $M_{jn}$
are weights depending on eigenfunctions $\varphi_n$ with corresponding
eigenvalues $\lambda_n$ that are characterized by a purely geometric
eigenvalue problem:
\begin{align*}
  \begin{cases}
    \begin{aligned}
    \Delta \varphi_n(\vx)\big)
      &= 0
      && \text{in }Y\setminus\Sigma~,
      \\[0.5em]
    \js{\varphi_n(\vx))} 
      &= 0
      && \text{on }\Sigma~,
      \\[0.5em]
      \lambda_n\js{\vn\cdot\nabla\varphi_n(\vx)}
      &=
      \nabla_T\cdot\nabla_T\varphi_n(\vx)
      && \text{on }\Sigma~.
    \end{aligned}
  \end{cases}\hspace{-2.75em}
\end{align*}
Many of the ideas developed in this publication can be adapted to this
eigenvalue problem. The stated goal is then to optimize the deviation of
$\eff_{ij}(\omega)$ to a target permittivity over a suitable frequency
band.


\section*{Acknowledgments}
M.B. and M.M. acknowledge partial support from the National Science
Foundation under grants DMS-1912846 and DMS-2045636.


\appendix
\section{Proofs of Lemmas \ref{lem:transformed_tensors_are_nice} and
\ref{lem:equivalence}, and Theorem~\ref{thm:dependence}}
\label{app:estimates}

\begin{proof}[Proof of Lemma \ref{lem:transformed_tensors_are_nice}]
  We know that $\hvq(\hvy)\in \vec D(\hat Y,\hat\Sigma)$ implying
  that $\hat J$ and $F(\hvy)$ are bounded with respect to
  $\|.\|_{L^\infty}$. Moreover, ${\eps}(\vx,\vy(\hvy)),
  {\sigma}(\vx,\vy(\hvy))$ are bounded, complex and
  tensor valued by assumption. Therefore, $\hat\eps(\vx,\hvy)$ and
  $\hat\sigma(\vx,\hvy)$ are by construction also bounded, complex and
  tensor valued. Simultaneous multiplication by a tensor from the left and
  its transpose from the right preserves symmetry, and so does scaling and
  multiplication by $\hat J$. Therefore, $\Re \hat\eps(\vx,\hvy)$,
  $\Im \hat\eps(\vx,\hvy)$, $\Re \hat\sigma(\vx,\hvy)$, and
  $\Im \hat\sigma(\vx,\hvy)$ are symmetric owing to the symmetry
  of $\eps(\vx,\vy)$ and $\sigma(\vx,\vy) $. Similarly, the uniform
  ellipticity of $\Im\hat\eps(\vx,\hvy)$ and $\Re\hat\sigma(\vx,\hvy)$ is a
  direct consequence of the uniform ellipticity of $\Im \eps(\vx,\vy)$ and
  $\Re \sigma(\vx,\vy)$ and the uniform lower bound on the determinant of
  the deformation gradient: $0<\delta<\hat J(\hvy)$.
\end{proof}

\begin{proof}[Proof of Lemma \ref{lem:equivalence}]
  We start with \eqref{eq:cell_problem_ale_bilinear}. Substituting
  $\sigma_{mn}$ given in \eqref{eq:transformed_tensors} and using
  Lemma~\ref{lem:tangential_transformation} for transforming $\vt_m$ and
  $\vt_n$ gives:
  \begin{align*}
    E(\hvchi,\hvvarphi;\hvq) &=
    \int_{\hat Y} \eps(\vx,\vy(\hvy))
    (\hat F(\hvy)^{-T} \hat F(\hvy)^{T}
    + \hat F(\hvy)^{-T}\hat\nabla\hvchi^{\:T})
    \cdot(\hat F(\hvy)^{-T}\overline{\hat\nabla\hvvarphi^{\:T}})\hat J\dhy\\
    &\quad - \frac{1}{i\omega} \sum\limits_{m,n}\int_{\hat \Sigma}
    \sigma_{mn}
    \Big(\frac{\hat F(\hvy)\hvt_m}{\|\hat F(\hvy) \hvt_m\|} \cdot \big(\hat
    F(\hvy)^{-T} \hat F(\hvy)^{T} + \hat F(\hvy)^{-T}
    \hat\nabla\hvchi^{\:T}\big)\Big)
    \\
    &\qquad\qquad\qquad\cdot\Big(
    \frac{\hat F(\hvy)\hvt_n}{\|\hat F(\hvy) \hvt_n\|} \cdot \hat
    F(\hvy)^{-T}\overline{\hat\nabla\hvvarphi^{\:T}}\Big)
    \|\hat F(\hvy)^{-T}\hvn\|_{\ell^2}\hat J \dhohy.
  \end{align*}
  After some more simplifications and using the definitions of $\hat
  \eps(\vx,\hvy)$ and $\hat \sigma_{mn}$ from equation
  \eqref{eq:transformed_tensors}:
  \begin{align*}
    E(\hvchi,\hvvarphi;\hvq) &=
    \int_{\hat Y} \hat\eps(\vx,\hvy)
    (\hat F(\hvy)^{T} + \hat\nabla\hvchi^{\:T})
    \cdot(\overline{\hat\nabla\hvvarphi^{\:T}})\hat J\dhy\\
    &\quad\    - \frac{1}{i\omega} \sum\limits_{m,n}\int_{\hat \Sigma}
    \hat \sigma_{mn}
    \Big( \hvt_m \cdot (\hat F(\hvy)^T+ \hat\nabla\hvchi^{\:T})\Big)
    \\
    &\qquad\qquad\qquad\quad\Big( \hvt_n \cdot \overline{\hat\nabla\hvvarphi^{\:T}}\Big)
    \|\hat F(\hvy)^{-T}\hvn\|_{\ell^2}\hat J \dhohy = 0.
  \end{align*}
  Now using the definition for $\hat\sigma(\vx,\hvy)$ given in
  \eqref{eq:transformed_tensors} yields \eqref{eq:cell_problem_ale_ng}.
  Equation \eqref{eq:eff_ale} can be transformed into \eqref{eq:eff_ale_ng}
  in a similar fashion.
\end{proof}

\begin{proof}[Proof of Theorem~\ref{thm:dependence}.]
  We begin with $E(\hvchi_1 - \hvchi_2, \hvchi_1 - \hvchi_2; \hvq_1)$.
  Setting $\hat F_1(\hvy) := \unit+\hat\nabla\hvq_1(\hvy)$, $\hat F_2(\hvy)
  := \unit+\hat\nabla\hvq_2(\hvy)$ and correspondingly, $\hat{J}_1$,
  $\hat{J}_2$, $\hat\eps_1(\vx, \hvy)$, $\hat\eps_2(\vx,\hvy)$,
  $\hat{\sigma}_1(\vx, \hvy)$, $\hat{\sigma}_2(\vx, \hvy)$, we observe that
  \begin{align*}
      E(\hvchi_1 - \hvchi_2, &\hvchi_1 - \hvchi_2; \hvq_1)
      \;=\; E(\hvchi_2,\hvchi_2-\hvchi_1; \hvq_1) -
      E(\hvchi_2,\hvchi_2-\hvchi_1; \hvq_2)
      \\
      &=  \int_{\hat Y} \big\{\hat\eps_1(\vx,\hvy) - \hat\eps_2(\vx,\hvy)\big\}
      \hat\nabla\hvchi_2^{\:T}\cdot
      \overline{\hat\nabla(\hvchi_2-\hvchi_1)^{\:T}}
      \dhy
      \\
      &\quad- \frac{1}{i\omega} \int_{\hat \Sigma}
      \big\{\hat\sigma_1(\vx,\hvy) - \hat\sigma_2(\vx,\hvy)\big\}
      \overline{\htnabla
      (\hvchi_2-\hvchi_2)^{\:T}\cdot\htnabla\hvchi_2^{\:T}} \dhohy
      \\
      &\quad+ \int_{\hat Y} \big\{\hat\eps_1(\vx,\hvy)\hat F_1(\hvy)^T -
      \hat\eps_2(\vx,\hvy)\hat F_2(\hvy)^T\big\}
      \cdot\overline{\hat\nabla(\hvchi_2-\hvchi_1)^{\:T}}\dhy
      \\
      &\quad- \frac{1}{i\omega} \int_{\hat \Sigma}
      \big\{\hat\sigma_1(\vx,\hvy)\tangent{(\hat F_1(\hvy))^T} -
      \hat\sigma_2(\vx,\hvy)\tangent{(\hat F_2(\hvy))^T}\big\}
      \cdot\overline{\htnabla(\hvchi_2-\hvchi_1)^{\:T}} \dhohy.
  \end{align*}
  Now, proceeding as in the proof of Theorem~\ref{thm:robustness}
  \begin{align*}
      \|\hat\nabla (\hvchi_1 &- \hvchi_2)\|_{L^2(\hat Y)}^2 +
      \frac{1}{\omega}
      \|\htnabla (\hvchi_1 - \hvchi_2)\|_{L^2(\hat \Sigma)}^2
      \\
      &\le
      C \Big\{
      \|\hat\eps_1(\vx,\hvy) - \hat\eps_2(\vx,\hvy)\|_{L^{\infty}(\hat Y)}
      \|\hat\nabla\hvchi_2\|_{L^2(\hat Y)}
      \|\hat\nabla(\hvchi_2-\hvchi_1)\|_{L^2(\hat Y)}
      \\
      &\qquad\quad+
      \frac1\omega\,
      \|\hat\sigma_1(\vx,\hvy) - \hat\sigma_2(\vx,\hvy)\|_{L^{\infty}(\hat\Sigma)}
      \|\htnabla\hvchi_2\|_{L^2(\hat \Sigma)}
      \|\htnabla(\hvchi_2-\hvchi_1)\|_{L^2(\hat \Sigma)}
      \\
      &\qquad\quad+
      \|\hat\eps_1(\vx,\hvy)\hat F_1(\hvy)^T -
      \hat\eps_2(\vx,\hvy)\hat F_2(\hvy)^T\|_{L^2(\hat Y)}
      \|\hat\nabla(\hvchi_2-\hvchi_1)\|_{L^2(\hat Y)}
      \\
      &\qquad\quad+ \frac1\omega\, \|\hat\sigma_1(\vx,\hvy)\hat F_1(\hvy)^T -
      \hat\sigma_2(\vx,\hvy)\hat F_2(\hvy)^T\|_{L^2(\hat \Sigma}
      \|\htnabla(\hvchi_2-\hvchi_1)\|_{L^2(\hat \Sigma)}
      \Big\}
  \end{align*}
  Young's inequality allows to absorb all factors with differences
  $\hvchi_2-\hvchi_1$ into the left hand side, and the factors
  $\|\hat\nabla\hvchi_2\|_{L^2(\hat Y)}$ and
  $\|\htnabla\hvchi_2\|_{L^2(\hat \Sigma)}$ can be bounded by a constant
  $C(\hvq_2)$ only depending on $\hvq_2$ (and the shape $\hat\Sigma$); see
  Theorem~\ref{thm:robustness}. In summary this implies that
  \begin{multline*}
    \|\hat\nabla (\hvchi_1 - \hvchi_2)\|_{L^2(\hat Y)}^2 +
    \frac{1}{\omega}
    \|\htnabla (\hvchi_1 - \hvchi_2)\|_{L^2(\hat \Sigma)}^2
    \\
    \le
    C(\hvq)\,\Big\{
    \|\hat\eps_1(\vx,\hvy) - \hat\eps_2(\vx,\hvy)\|_{L^{\infty}(\hat
    Y)}^2
    +
    \frac1\omega\, \|\hat\sigma_1(\vx,\hvy) -
    \hat\sigma_2(\vx,\hvy)\|_{L^{\infty}(\hat\Sigma)}^2
    \Big\}.
  \end{multline*}
  The final inequality now follows from the fact that for a fixed
  coordinate $\hvy$ the tensors $\varepsilon$ and $\sigma$ depend
  analytically on $\hvq$; see \eqref{eq:transformed_tensors}.
\end{proof}


\section{Proof of Theorem \ref{thm:regularity}}
\label{app:regularity}
In order to show regularity we follow the well-established strategy of
introducing a difference quotient and then passing to the limit; see for
example \cite{gilbarg2001}. For the sake of completeness, we restate and
generalize the argument here so that it can be applied to our case of a
periodic domain with a curved hypersurface. A number of preparatory steps
are in order.
\begin{definition}
  \label{def:diff_quotient}
  Let $\veta(x): \hat Y \to \R^n$ be a smooth, $\hat Y$-periodic vector
  field with $\veta(\hvy)\cdot \vec n(\hvy) = 0$ for $\hvy\in \hat\Sigma$,
  where $\vec n$ denotes a fixed normal field on $\hat\Sigma$, and assume
  that $\veta(\hvy)=\vec 0$ for $\hvy\in\partial\hat\Sigma$.
  \begin{itemize}
    \item[a)]
      For $\hvy\in\hat Y$ let
      $\vec\xi_{\veta,\hvy}\,:\,\mathbb{R}\to\hat Y$ be the integral curve
      of $\veta$ with $\vec\xi_{\veta,\hvy}(0)=\hvy$ and
      $\tfrac{\text{d}}{\text{d}s}\vec\xi_{\veta,\hvy}(s) =
      \veta(\vec\xi_{\veta,\hvy}(s))$ for all $s\in\mathbb{R}$. Here, by
      slight abuse of notation, we equip $\hat Y$ with a toroidal topology
      by identifying opposing periodic boundaries, so that
      $\vec\xi_{\veta,\hvy}$ admits
      $\mathbb{R}$ as domain of definition.
    \item[b)]
      Introduce a transformation, $T_h\,:\,\hat Y\to\hat Y$ characterized
      by $T_h(\hvy)\,:=\,\vec\xi_{\veta,\hvy}(h)$.
    \item[c)]
      For a given $f \in C^\infty_{\text{per}}(\hat Y)$ and $h>0$ introduce
      two difference operators:
      \begin{equation*}
        \hat\nabla_\eta^h f(\hvy) := \frac{f\big(T_h(\hvy)\big)-
        f(\hvy)}{h},
        \qquad
        \tnegdiffop f(\hvy) := \frac{f(\hvy) - f\big(T_{-h}(\hvy)\big)
        \kappa(\hvy)}{h}.
      \end{equation*}
      Here, $\kappa(\hvy) := \text{det}\big(\nabla T_{-h}(\hvy)\big)$.
  \end{itemize}
\end{definition}
We have the following results at hand:
\begin{lemma}
  \label{lem:diff_op_lemma}
  Let $\veta$ be a vector field as characterized in
  definition~\ref{def:diff_quotient}. Then, for all $f,g \in
  C^{\infty}_{\text{per}}(\hat Y)$ and under the assumption that $h>0$ is
  sufficiently small it holds that
  \begin{align*}
    \begin{aligned}
      &\text{(1.)} &\quad&
      T_h\text{ and }T_{-h}\text{ are automorphisms on }\hat Y\text{
      and } \hat\Sigma,\text{ with }T_{-h}\circ T_h = \text{Id},
      \\
      &\text{(2.)} &\quad&
      |\nabla\kappa(\hvy)|\le c\,h\,
      \max\big(|\nabla\veta|_\infty(\hvy),|\nabla^2\veta(\hvy)|_\infty\big)
      \text{ for all }\hvy\in\hat Y,
      \\
      &\text{(3.)} &\quad&
      \int_{\hat Y} (\tdiffop f(\hvy)) g(\hvy) \dhy
      = -\int_{\hat Y} f(\hvy) (\tnegdiffop g(\hvy)) \dhy,
      \\
      &\text{(4.)} &\quad&
      \int_{\hat \Sigma} (\tdiffop f(\hvy)) g(\hvy) \dhohy
      = -\int_{\hat \Sigma} f(\hvy) (\tnegdiffop g(\hvy)) \dhohy,
      \\
      &\text{(5.)} &\quad&
      \tdiffop(fg)(\hvy) = f(T_h(\hvy))(\tdiffop g)(\hvy)
      +(\tdiffop f)(\hvy) g(\hvy),
      \\
      &\text{(6.)} &\quad&
      \lim\limits_{h\to0} \tdiffop f(\hvy) = \veta(\hvy)\cdot\hat\nabla f(\hvy)
      =:\hat\nabla_{\eta}f(\hvy).
    \end{aligned}
  \end{align*}
\end{lemma}

\begin{proof}
  (1.) The integral curve $\vec\xi_{\veta,\hvy}(s)$ is described by an
  initial value problem with smooth right hand side on $\hat Y$. Thus, by
  virtue of the Picard-Lindelöf theorem a unique solution $\vec\xi_{\veta,\hvy}(s)$ exists.
  Moreover, $\vec\xi_{\veta,\hvy}(s)$ depends
  smoothly on the initial data $\hvy$ and the differential
  $\nabla_{\hvy}\vec\xi_{\veta,\hvy}(s)$ is given by an initial value
  problem
  \begin{align}
    \label{eq:ivp}
    \nabla_{\hvy}\vec\xi_{\veta,\hvy}(0)=\unit,
    \qquad
    \frac{\text{d}}{\text{d}s}\nabla_{\hvy}\vec\xi_{\veta,\hvy}(s)
    =
    \nabla\veta(\xi_{\veta,\hvy}(s))
    \nabla_{\hvy}\vec\xi_{\veta,\hvy}(s).
  \end{align}
  This shows that the transformation $T_h(\hvy)=\vec\xi_{\veta,\hvy}(h)$ is
  a well-defined differentiable function. Moreover, owing to the
  compactness of $\hat Y$ there exists an $h_0$ such that $\|\nabla
  T_h(\hvy) - \unit\|\le\delta<1$ for all $\hvy\in\hat Y$ and $|h|\leq h_0$.
  This implies that $T_h$ is injective for $|h|\leq h_0$. A consequence of
  the initial value problem of the integral curves is the fact that
  $\vec\xi_{\veta,\vec{\tilde y}}(-h)=\hvy$ for $\vec{\tilde
  y}=\vec\xi_{\veta,\hvy}(h)$ for all $\hvy\in\hat Y$. This implies that
  $T_{-h}\circ T_{h}=Id$.

  (2.) We first observe that $\nabla\kappa(\hvy)$ is a sum of products
  of $(d-1)$ entries of $\nabla_{\hvy}\vec\xi_{\veta,\hvy}(h)$ and one
  entry of the tensor of second derivatives
  $\nabla^2_{\hvy}\vec\xi_{\veta,\hvy}(h)$. This implies that
  \begin{align*}
    |\nabla\kappa(\hvy)|_{\infty} \le
    c\,|\nabla_{\hvy}\vec\xi_{\veta,\hvy}(h)|_{\infty}^{d-1}\,
    |\nabla^2_{\hvy}\vec\xi_{\veta,\hvy}(h)|_{\infty}
    \le
    c\, |\nabla^2_{\hvy}\vec\xi_{\veta,\hvy}(h)|_{\infty},
  \end{align*}
  where for the second inequality we increased the constant $c$ with a
  uniform bound on $\nabla_{\hvy}\vec\xi_{\veta,\hvy}(h)$ that was
  established for $|h|\le h_0$ in (1.). We now observe that
  $\nabla^2_{\hvy}\hvy_{\veta,\hvy}(0)=\vec 0$. The tensor of second
  derivatives obeys an initial value problem similarly to \eqref{eq:ivp}
  but with a right hand side involving $\nabla\veta(\hvy)$ and
  $\nabla^2\veta(\hvy)$:
  \begin{align*}
    \frac{\text{d}}{\text{d}s}\nabla^2_{\hvy}\vec\xi_{\veta,\hvy}(s)
    =
    (\nabla^2\veta(\xi_{\veta,\hvy}(s))
    \nabla_{\hvy}\vec\xi_{\veta,\hvy}(s)\big)\cdot
    \nabla_{\hvy}\vec\xi_{\veta,\hvy}(s) +
    \nabla\veta(\xi_{\veta,\hvy}(s))
    \nabla^2_{\hvy}\vec\xi_{\veta,\hvy}(s).
  \end{align*}
  Integrating the differential equation and using the initial condition
  we get
  \begin{align*}
    |\nabla^2_{\hvy}\vec\xi_{\veta,\hvy}(h)|_{\infty}
    \,\le\, c\,
    \int_0^h\max\Big(|\nabla\veta|,|\nabla^2\veta|\Big)
    (\vec\xi_{\veta,\hvy}(s))\text{d}s.
  \end{align*}
  possibly shrinking $h_0$ again with a compactness argument now
  establishes
  \begin{align*}
    |\nabla^2_{\hvy}\vec\xi_{\veta,\hvy}(h)|_{\infty}
    \,\le\,
    ch\,\max\big(|\nabla\veta|_\infty(\hvy),|\nabla^2\veta(\hvy)|_\infty\big).
  \end{align*}

  (3.) By definition,
  \begin{align*}
  \int_{\hat Y} (\tdiffop f(\hvy)) g(\hvy) \dhy
  &= \int_{\hat Y} \frac{f(T_h(\hvy))}{h}g(\hvy)\dhy- \int_{\hat
  Y}\frac{f(\hvy)}{h} g(\hvy) \dhy
  \intertext{Applying the transformation $\hvy \mapsto T_{-h}(\hvy)$
    and exploiting $\hat Y$-periodicity gives:}
  &=\int_{\hat Y} \frac{f(\hvy)}{h}g(T_{-h}(\hvy))\kappa(\hvy)
  \dhy- \int_{\hat Y}\frac{f(\hvy)}{h} g(\hvy) \dhy
  \\
  &= -\int_{\hat Y} f(\hvy) (\tnegdiffop g(\hvy)) \dhy.
  \end{align*}

  (4.) The proof of this statement is similar to (1.) with the important
  detail that we have to establish that the transformed surface element is
  given by $\kappa(\hvy)\dhohy$. By definition of $\veta$ we have
  $\veta(\hvy)\cdot \vec n(\hvy) = 0$ for $\hvy \in \hat\Sigma$. This
  implies that $\hat\Sigma$ is \emph{parallel} to integral curves and as a
  consequence we have that
  \begin{align*}
    \hat\nabla T_h(\hvy) \simeq
    \begin{pmatrix}
      \partial_{\tau_1} (T_h\cdot\vec \tau_1) &
      \partial_{\tau_2} (T_h\cdot\vec \tau_1)&
      \partial_{n} (T_h\cdot\vec \tau_1)\\
      \partial_{\tau_1} (T_h\cdot\vec \tau_2) &
      \partial_{\tau_2} (T_h\cdot\vec \tau_2) &
      \partial_{n} (T_h\cdot\vec \tau_2)\\
      0 & 0 & 1
    \end{pmatrix},
  \end{align*}
  when expressing the Jacobian of $T_h(\hvy)$ for $\hvy\in\hat\Sigma$ in a
  local coordinate system spanned by $(\vec n, \vec\tau_1, \vec\tau_2)$.

  (5.) An elementary calculation shows:
  \begin{align*}
    \tdiffop(fg)(\hvy)
    &= \frac{f(T_h(\hvy))g(T_h(\hvy)) -
    f(T_h(\hvy))g(\hvy)}{h}
    \\
    &\qquad\qquad\qquad\qquad\qquad\qquad
    +\frac{f(T_h(\hvy))g(\hvy) - f(\hvy)g(\hvy)}{h}
    \\
    &= f(T_h(\hvy))(\tdiffop g)(\hvy) + g(\hvy) (\tdiffop f)(\hvy).
  \end{align*}

  (6.) This is an immediate consequence of the differentiability of
  $f(\hvy)$ and the L'Hôpital theorem.
\end{proof}
With this definition and lemma, we are in a position to prove the following
intermediate result:
\begin{proposition}
\label{thm:regularity_intermediate}
  Let $\veta(x): \hat Y \to \R^n$ be a smooth, $\hat Y$-periodic vector
  field with $\veta(\hvy)\cdot \vec n(\hvy) = 0$ for $\hvy\in \hat\Sigma$,
  as well as $\veta(\hvy)=\vec 0$ for all $\hvy\in\partial\hat\Sigma$. Let
  $\hvchi \in H$ be the solution to \eqref{eq:cell_problem}. Then, provided
  that $\eps(\vx,\hvy)$, $\sigma(\vx,\hvy)$ and $\hvq(\hvy)$ are
  sufficiently regular:
  \begin{multline}
    \label{eq:estimate}
    \|\hat\nabla \hat\nabla_\eta\hvchi\|^2_{L^2(\hat Y)}
    +\frac1\omega\|\htnabla \hat\nabla_\eta\hvchi\|^2_{L^2(\hat\Sigma)}
    \le\;C\;\max\Big\{1,\|\hat\nabla\veta\|_{L^\infty(\hat Y)},
    \|\hat\nabla^2\veta\|_{L^\infty(\hat Y)}\Big\}
    \\
    \times
    \Big\{
    \|\hat\eps(\vx,\hvy)\|^2_{W^{1,\infty}(\hat Y)}
    \|\hat\nabla\hvchi\|^2_{L^2(\hat Y)}
    + \frac{1}{\omega}
    \|\hat \sigma(\vx,\hvy)\|^2_{W^{1,\infty}(\hat\Sigma)}
    \|\htnabla\hvchi\|^2_{L^2(\hat\Sigma)}
    \\
    +\|\hat\eps(\vx,\hvy)\|^2_{W^{1,\infty}(\hat Y)}
    \|\hat F^T \vec e_i\|^2_{H^1(\hat Y)}
    + \frac{1}{\omega}
    \|\hat \sigma(\vx,\hvy)\|^2_{W^{1,\infty}(\hat\Sigma)}
    \|\hat F^T \vec e_i\|^2_{H^1(\hat\Sigma)}\Big\}.
  \end{multline}
\end{proposition}
\begin{proof}
  Above assumptions on $\veta(\hvy)$ ensure that $\hvvarphi := \tnegdiffop
  \tdiffop \;\hvchi\,\in H$. Testing \eqref{eq:cell_problem_ale}
  with this choice of test function $\hvvarphi$ and taking the real part:
  \begin{multline}
    \label{eq:intermediate}
    \int_{\hat Y} \Re\hat\eps\,\hat\nabla\hvchi^{\:T}
    \cdot\hat\nabla{\tnegdiffop \tdiffop \overline{\hvchi}^{\:T}}\dhy
    + \frac{1}{\omega} \int_{\hat \Sigma}\Im\hat\sigma\,
    \htnabla \hvchi^{\:T}\cdot\htnabla
    \tnegdiffop \tdiffop \overline{\hvchi}^{\:T} \dhohy
    \\
    = - \Re\Big\{\int_{\hat Y} \hat\eps(\vx,\hvy)
    (\hat F(\hvy)^T) \cdot\hat\nabla
    \tnegdiffop \tdiffop \overline{\hvchi}^{\:T}\dhy \\
    + \frac{1}{i\omega} \int_{\hat \Sigma} \hat\sigma(\vx,\hvy)
    \htangent{(\hat F(\hvy)^T)}\cdot\htnabla
    \tnegdiffop \tdiffop \overline{\hvchi}^{\:T} \dhohy\Big\}.
  \end{multline}
  We now wish to move the difference operator $\tnegdiffop$ from the
  testfunction over to $\hvchi(\hvq)$ and the forcing terms but we are
  faced with the issue that due to the curvature encoded in $\veta(\hvy)$
  the operators $\hat\nabla$, $\htnabla$, $\tnegdiffop$, $\tdiffop$ do not
  necessarily commute. Observe, for example, that
  \begin{align}
    \label{eq:jump_example}
    [\hat\nabla,\tnegdiffop] f(\hvy) \;=\;
    \hat\nabla f\big(T_{-h}(\hvy)\big) \hat\nabla \veta(\hvy)\kappa(\hvy)
    + f\big(T_{-h}(\hvy)\big)\frac{\hat\nabla \kappa(\hvy)}{h},
  \end{align}
  where $[.,.]$ denotes the commutator. Lemma~\ref{lem:diff_op_lemma}(2.)
  establishes that $|\kappa(\hvy)/h|_\infty \le
  c\,\max\big(|\nabla^2\veta(\hvy)|_\infty,
  |\nabla^2\veta(\hvy)|_\infty\big)$. Applying this result to
  \eqref{eq:jump_example} and taking the limit $h\to0$ on the right
  hand side yields
  \begin{align}
    \label{eq:jump_estimate}
    \big\|[\hat\nabla,\tnegdiffop] f(\hvy)\big\|_{L^2(\hat Y)} \,\le\, C\,
    \max\big\{1,\|\hat\nabla\veta\|_{L^\infty(\hat
    Y)},\|\hat\nabla^2\veta\|_{L^\infty(\hat Y)}\big\}\,\big\|\hat\nabla
    f(\hvy)\|_{L^2(\hat Y)},
  \end{align}
  where we have used a Poincaré inequality valid for periodic $f(\hvy)$. A
  similar result holds for all other commutator pairings. Then, moving the
  difference operator and applying the integration by parts formula and
  product rule formula established in Lemma~\ref{lem:diff_op_lemma} we
  arrive at
  \begin{multline}
    \label{eq:intermediate2}
    \int_{\hat Y} \Re\hat\eps(\vec x, T_h(\hvy))\,\hat\nabla\tdiffop
    \hvchi^{\:T} \cdot\hat\nabla{\tdiffop \overline{\hvchi}^{\:T}} +
    (\tdiffop \Re\hat\eps)\,\hat\nabla\hvchi^{\:T} \cdot\hat\nabla{\tdiffop
    \overline{\hvchi}^{\:T}}\dhy
    \\
    + \frac{1}{\omega} \int_{\hat \Sigma}\Im\hat\sigma(\vec x, T_h(\hvy))
    \htnabla\tdiffop \hvchi^{\:T}\cdot\htnabla
    \tdiffop \overline{\hvchi}^{\:T}+ (\tdiffop \Im\hat\sigma)
    \htnabla\hvchi^{\:T}\cdot\htnabla
    \tdiffop \overline{\hvchi}^{\:T} \dhohy
    \\
    = - \Re\Big\{\int_{\hat Y} \tdiffop\big\{\hat\eps(\vx,\hvy)
    (\hat F(\hvy)^T)\big\} \cdot\hat\nabla
    \tdiffop \overline{\hvchi}^{\:T}\dhy \\
    + \frac{1}{i\omega} \int_{\hat \Sigma} \tdiffop\big\{\hat\sigma(\vx,\hvy)
    \htangent{(\hat F(\hvy)^T)}\big\}\cdot\htnabla
    \tdiffop \overline{\hvchi}^{\:T} \dhohy\Big\}
    \\
    +\;\big\{\text{commutator terms}\big\},
  \end{multline}
  where we have collected all commutator terms in the last term (and
  discuss them further down below). Proceeding again as in the proof for
  Theorem~\ref{thm:robustness} by using Young's inequality for all terms on
  the right hand side and uniform ellipticity of $\Re\hat\eps$ and
  $\Im\hat\sigma$ yields:
  \begin{multline}
    \label{eq:intermediate3}
    \|\hat\nabla \tdiffop\hvchi\|^2_{L^2(\hat Y)}
    +\frac1\omega\|\htnabla \tdiffop\hvchi\|^2_{L^2(\hat\Sigma)}
    \\
    \le\;C\;\Big\{
    \|\hat\eps(\vx,\hvy)\|^2_{W^{1,\infty}(\hat Y)}
    \|\hat\nabla\hvchi\|^2_{L^2(\hat Y)}
    + \frac{1}{\omega}
    \|\hat \sigma(\vx,\hvy)\|^2_{W^{1,\infty}(\hat\Sigma)}
    \|\htnabla\hvchi\|^2_{L^2(\hat\Sigma)}
    \\
    +\|\hat\eps(\vx,\hvy)\|^2_{W^{1,\infty}(\hat Y)}
    \|\hat F^T\|^2_{H^1(\hat Y)}
    + \frac{1}{\omega}
    \|\hat \sigma(\vx,\hvy)\|^2_{W^{1,\infty}(\hat\Sigma)}
    \|\hat F^T\|^2_{H^1(\hat\Sigma)}
    \\
    +\;\big|\text{commutator terms}\big|
    \Big\}
  \end{multline}
  In the estimate above we have passed to the limit, $h\to0$, on the right
  hand side.

  As a last step we will discuss the commutator terms. The volume integral
  over $\hat Y$, for example, gives rise to the following terms:
  \begin{multline*}
    \text{(a)} \;=\;
    \int_{\hat Y} \Re\hat\eps\,\hat\nabla\hvchi^{\:T}
    \cdot[\hat\nabla,\tnegdiffop] \tdiffop \overline{\hvchi}^{\:T}\dhy
    \\
    \;+\;
    \int_{\hat Y} \Re\hat\eps(T_h(\hvy))\,[\hat\nabla,\tdiffop] \hvchi^{\:T}
    \cdot\hat\nabla{\tdiffop \overline{\hvchi}^{\:T}}\dhy
    \\
    + \Re\Big\{\int_{\hat Y} \hat\eps(\vx,\hvy)
    (\hat F(\hvy)^T) \cdot
    [\hat\nabla,\tnegdiffop] \tdiffop \overline{\hvchi}^{\:T}\dhy\Big\}.
  \end{multline*}
  Applying \eqref{eq:jump_estimate} allows us to estimate $\text{(a)}$ by
  \begin{multline*}
    \big|(a)\big| \;\le\; C\,
    \|\hat\eps(\vx,\hvy)\|_{L^{\infty}(\hat Y)}
    \max\big\{1,\|\hat\nabla\veta\|_{L^\infty(\hat
    Y)},\|\hat\nabla^2\veta\|_{L^\infty(\hat Y)}\big\}
    \\
    \|\hat\nabla \tdiffop\hvchi\|_{L^2(\hat Y)}
    \;\Big\{ \|\hat\nabla\hvchi\|_{L^2(\hat Y)}
    + \|\hat F^T\|_{H^1(\hat Y)} \Big\}.
  \end{multline*}
  This expression can again be absorbed into the remaining terms of
  \eqref{eq:intermediate3} by changing the constant $C$ to $C
  \max\{1,\|\hat\nabla\veta\|_{L^\infty(\hat Y)},
  \|\hat\nabla^2\veta\|_{L^\infty(\hat Y)}\}$. An analogous result holds
  for all terms arising from the surface integral over $\hat\Sigma$.
  Passing to the limit $h\to0$ on the left side \cite{gilbarg2001} now
  shows the final estimate given in \eqref{eq:estimate}.
\end{proof}
\begin{proof}[Proof of Theorem~\ref{thm:regularity}]
  Estimate \eqref{eq:regularity} is an immediate consequence of
  Proposition~\ref{thm:regularity_intermediate} by setting
  $\vec\eta=\hat{\vec\tau}_i$ and using Theorem~\ref{thm:robustness} to
  estimate the $\|\hat\nabla\hvchi\|_{L^2(\hat Y)}$ and
  $\|\htnabla\hvchi\|_{L^2(\hat \Sigma)}$ terms on the right hand side of
  \eqref{eq:estimate}.
\end{proof}


\bibliographystyle{siamplain}

\end{document}